\documentclass[11pt,reqno]{amsart}
\usepackage{newlfont,
amsmath,amsthm,amsgen,amscd}
\usepackage{amssymb}
\usepackage{latexsym}
\usepackage{fancyhdr}
\usepackage{ifthen}
\usepackage{amsfonts}
\usepackage{lastpage}
\usepackage{afterpage}
\usepackage{epsfig}
\usepackage{enumerate}
\usepackage{euscript,color}

\newcommand{\belabel}[1]{\begin{equation}\label{#1}}
\newtheorem{mthm}{Theorem}

\newcommand{\vf}{\varphi}
\newcommand{\C}{{\mathbb C}}

\newcommand{\rr}{\mathbb{R}}

\newcommand{\CC}{{\mathbb C}}

\newcommand{\II}{{\mathbb I}}

\newcommand{\RR}{{\mathbb R}}

\newcommand{\bM}{{\overline{\cal M}}}
\newcommand{\bS}{\overline{S}}
\newcommand{\bg}{{\overline{\mathbf{g}}}}
\newcommand{\bnab}{{\overline{\nabla}}}
\newcommand{\bR}{{\overline{\mathrm{R}}}}
\newcommand{\R}{{\mathrm{R}}}
\newcommand{\bRic}{{\overline{\mathrm{Ric}}}}
\newcommand{\bscal}{{\overline{\mathrm{scal}}}}

\newcommand{\B}{{\mathrm{B}}}

\newcommand{\fX}{{\mathfrak{X}}}
\newcommand{\grad}{{\mathrm{grad}}}
\newcommand{\Hess}{{\mathrm{Hess}}}
\newcommand{\diver}{{\mathrm{div}}}
\renewcommand{\div}{{\mathrm{div}}}
\newcommand{\tr}{\mathrm{tr}}

\newcommand{\prT}{\mathrm{pr}_{T^\perp}}
\newcommand{\cL}{{\cal L}}
\newcommand{\Id}{{\mathrm{Id}}}
\newcommand{\F}{{\cal F}}
\newcommand{\beq}{\begin{eqnarray*}}
\newcommand{\eeq}{\end{eqnarray*}}
\newcommand{\be}{\begin{eqnarray}}
\newcommand{\ee}{\end{eqnarray}}
\newcommand{\Ric}{{\mathrm{Ric}}}
\newcommand{\beqn}{\begin{equation}}
\newcommand{\eeqn}{\end{equation}}

\theoremstyle{definition}

\newtheorem{re}{Remark}[section]

\newtheorem{bsp}{Example}[section]

\newtheorem*{bsp*}{Example}
\newtheorem*{def*}{Definition}
\theoremstyle{plain}
\newtheorem{Lemma}{Lemma}[section]
\newtheorem*{lem*}{Lemma}
\newtheorem{Proposition}{Proposition}[section]
\newtheorem{Corollary}{Corollary}[section]
\newtheorem{Theorem}{Theorem}[section]
\newtheorem*{theo*}{Theorem}
\newtheorem*{conj*}{Conjecture}

\newcommand{\cU}{{\cal U}}
\newcommand{\M}{{\cal M}}
\newcommand{\cF}{{\cal F}}
\newcommand{\g}{\mathbf{g}}
\renewcommand{\II}{\mathrm{II}}
\newcommand{\W}{\mathrm{W}}
\newcommand{\A}{\mathrm{A}}
\newcommand{\T}{\mathrm{T}}
\newcommand{\scal}{\mathrm{scal}}
\newcommand{\del}{\partial}

\newcommand{\bleml}[1]{\begin{Lemma} \label{#1}}
\newcommand{\blem}{\begin{Lemma}}
\newcommand{\elem}{\end{Lemma}}
\newcommand{\Lem}[1]{Lemma~\ref{#1}}

\newcommand{\btheo}{\begin{Theorem}}
\newcommand{\btheol}[1]{\begin{Theorem}\label{#1}}
\newcommand{\etheo}{\end{Theorem}}

\newcommand{\bpropl}[1]{\begin{Proposition} \label{#1}}
\newcommand{\bprop}{\begin{Proposition}}
\newcommand{\eprop}{\end{Proposition}}
\newcommand{\Prop}[1]{Proposition~\ref{#1}}

\newcommand{\bcorl}[1]{\begin{Corollary} \label{#1}}
\newcommand{\bcor}{\begin{Corollary}}
\newcommand{\ecor}{\end{Corollary}}

\newcommand{\bbem}{\begin{re}}
\newcommand{\ebem}{\end{re}}

\newcommand{\bprf}{\begin{proof}}
\newcommand{\eprf}{\end{proof}}
\newcommand{\pr}{\mathrm{pr}}

\begin{document}
\title[Cauchy problems for Lorentzian manifolds with special
holonomy]{Cauchy problems for Lorentzian manifolds\\[1mm]with special
holonomy}

\subjclass[2010]{Primary  53C50, 53C27; Secondary  53C44, 35A10, 83C05}
\keywords{Lorentzian manifolds, special holonomy, parallel spinors, parallel null vectors, generalised Killing spinors, Cauchy problem}

\author{Helga Baum}
\address[Baum \& Lischewski]{Humboldt-Universit\"{a}t zu Berlin,
Institut f\"{u}r Mathematik, Rudower Chaus\-see~25, 12489~Berlin, Germany}\email{baum@math.hu-berlin.de \& lischews@math.hu-berlin.de}
\author{Thomas Leistner}\address[Leistner]{School of Mathematical Sciences, University of Adelaide, SA~5005,
Australia} \email{thomas.leistner@adelaide.edu.au}

\author{Andree Lischewski} 
\thanks{The
 authors acknowledge support from the Australian Research
Council through the grants FT110100429 and DP120104582 and from the Collaborative Research Centre
647 ``Space-Time-Matter'' of the German Research Foundation.}
\begin{abstract}
On a  Lorentzian manifold the existence of a parallel null vector field implies certain constraint conditions on the induced Riemannian geometry of a space-like hypersurface. We will derive these constraint conditions and, conversely, show that every real analytic Riemannian manifold satisfying the constraint conditions can be extended to a Lorentzian manifold  with a parallel null vector field. Similarly, every parallel null spinor on a Lorentzian manifold induces an imaginary generalised Killing spinor on a space-like hypersurface. Then, based on  the fact that a parallel  spinor field induces a parallel  vector field, we can apply the first result to prove: every real analytic Riemannian manifold carrying a real analytic, imaginary generalised Killing spinor can be extended to a Lorentzian manifold with a parallel null spinor. Finally, we give examples of geodesically complete Riemannian manifolds satisfying the constraint conditions.
\end{abstract}

\maketitle


\section{Background and  main results}

This paper is a contribution to the research programme of studying global and causal properties of Lorentzian manifolds with special holonomy. A Lorentzian manifold has {\em special holonomy} if  the connected component of its holonomy group is reduced from the full group $\mathbf{SO}^0(1,n)$, but still acts indecomposably, i.e., without non-degenerate invariant subspaces. In this situation the Lorentzian manifold admits a  bundle of tangent null lines that is invariant  under parallel transport. The possible special Lorentzian holonomy groups were classified in \cite{bb-ike93} and \cite{leistnerjdg}, all of them can be realised by local metrics \cite{galaev05}, but many questions about the consequences of special holonomy for global and causal properties of the manifold are still open.  A special case of this situation is when the parallel null line bundle is spanned by a {\em parallel null vector field}. This is the case we will study in this paper. It  is motivated by the question which Lorentzian manifolds admit a {\em parallel spinor field}, which in turn draws its motivation from mathematical physics.
Since a parallel spinor is invariant under the spin representation of the holonomy group, indecomposable Lorentzian manifolds with parallel spinors have special holonomy.
However, since $\mathbf{SO}^0(1,n)$ has no proper {\em irreducible} subgroups, the situation is very different from the Riemannian case, where we have several  {\em irreducible} holonomy groups that admit an invariant spinor. In fact, a  spinor field $\phi$ on any Lorentzian manifold $(\bM,\bg)$ induces a causal vector field $V_\phi$, its {\em Dirac current}, which is defined by
\[\bg(X,V_\phi)=-\langle X\cdot \phi,\phi\rangle,\]
for all $X\in T\bM$. 
If the spinor $\phi$ is parallel, $V_\phi$ is a parallel vector field and thus reduces the holonomy to its stabiliser. Moreover, if we assume that the manifold is indecomposable, $V_\phi$  must be  {\em null}, by which we mean $\bg(V_\phi,V_\phi)=0$ and $V_\phi\not=0$.

A fundamental problem in the programme mentioned at the beginning is the question: What is the intersection of the class of globally
hyperbolic Lorentzian manifolds with the class of Lorentzian manifolds with special holonomy?  A Lorentzian manifold is {\em globally hyperbolic} if it admits a Cauchy hypersurface, i.e., a space-like hypersurface that is met by every inextendible timelike curve exactly once. 
It is known from work of Bernal and S\'{a}nchez 
\cite{bernal-sanchez03}
that a globally hyperbolic Lorentzian manifold $(\bM,\bg)$ is of the form $\bM=\rr\times \M$ with the metric
\belabel{globhypmetric}
\bg=-\lambda dt^2+\g_t,\eeqn
where  $\g_t$ is a $t$-dependent family of Riemannian
metrics on $\M$ and $\lambda=\lambda(t,x) $ is a smooth function on
$\bM$, the so-called {\em lapse function}.
Hence, the first step in order to understand globally hyperbolic manifolds with special holonomy, and more specifically, with parallel null vector field or parallel null spinor field (for which $V_\phi$ is null) is to  investigate the following questions:
\begin{enumerate}[(A)]
\item
\label{cc}
{\em Constraint conditions:} What are the  {\em constraint conditions} that are imposed on the space-like hypersurace $\M$ in the manifold  $(\bM,\bg)$ by the existence of  
{\em (i) parallel null vector field}, or {\em (ii)
a  parallel null spinor field}?
\item\label{cp} {\em Cauchy problem:}
Can we extend a  given a Riemannian manifold satisfying these constraint conditions to a Lorentzian manifold with metric  as in \eqref{globhypmetric} with a
{\em (i) parallel null vector field}, or {\em (ii)
a  parallel null spinor field}?
\end{enumerate}
By evaluating the Gau\ss-Codazzi equations we find the answer to question \ref{cc}(i):
If $(\bM,\bg)$ admits a parallel null vector field $V$ then  there is a vector field $U=-\pr_{T\M}(V)$ such that
\be\label{constraint-intro}
\nabla^\g U+u\W&=&0,
\ee
with $u^2=\g(U,U)$ and in which $\W:=-\bnab T$ is the Weingarten operator of $\M\subset (\bM,\bg)$.
Hence question \ref{cp}(i) can be stated more precisely as: 
 can a  Riemannian manifold $(\M,\g)$ that is given together with  a vector field $U$ and a $\g$-symmetric endomorphism field $\W$  satisfying the constraint equation \eqref{constraint-intro} be extended to  Lorentzian manifold $(\bM,\bg)$ as in \eqref{globhypmetric} with a parallel null vector field $V$ that projects to $U$.  In Sections  \ref{section-par-vf} and \ref{analytic} we will derive  a PDE system in the form of {\em evolution equations} and which
 is equivalent to the existence of a parallel null vector field $V$ for the metric in \eqref{globhypmetric}.
It is of the form
 \belabel{pdesys}
 \del^2_t\cal V =F( \cal V, \del_i\cal V,\del_t\cal V, \del_i\del_j\cal V, \del_i\del_t\cal V)
 \eeqn
 with $\cal V(t,x^i)=(\g(t,x^i), U(t,x^i), u(t,x^i)) $ a triple of symmetric bilinear forms, vectors and functions depending on $t$ and $x^i$ (see Theorems \ref{Theorem-parallel-vf} and 
\ref{Theorem-parallel-vf2} for a detailed statement).
 Thus we can apply
 the Cauchy-Kowalevski Theorem to \eqref{pdesys}, provided that the initial data and the lapse function  are analytic, and obtain:
 \begin{mthm}\label{maintheo1}
 Let
$(\M,\g, \W, U)$ be an analytic Riemannian manifold together with
a field of $\g$-symmetric, analytic endomorphisms $\W$,  and an analytic
vector field $U$ 
satisfying the  constraint equation
\eqref{constraint-intro}. Then, for any analytic
 function  $\lambda$ on $\RR\times \M$ there exists an open neighbourhood $\overline{\cU}$ of $\{0\}\times \M$ in $
 \RR \times \M$  and an unique analytic Lorentzian metric
\[ \bg=-\lambda^2dt^2 +\g_t\]
on $\overline{\cU}$ which admits an analytic, 
 parallel  null vector field  $\;V=\frac{u_t}{\lambda} \del_t-U_t$, with analytic  $t$-dependent families of Riemannian metrics
 $\g_t$,  vector fields $U_t$ and functions $u_t$  on $\M$
  satisfying the initial conditions
 $\g_0=\g$, $U_0=U$, $u_0=u$, and
\beq
\dot \g_0&=& -2\lambda_0 \II,
\\
\dot U_0&=& u \ \grad^\g(\lambda_0)   +\lambda_0\W(U),
\\
\dot u_0&=&d\lambda_0(U).
\eeq
\end{mthm}
As an illustration, in Proposition \ref{Prop-example} we provide an example  in which $\W$ is a Codazzi tensor on $(\M,\g)$ and for which we can explicitly solve the corresponding system \eqref{pdesys} for a constant function $\lambda$.

In Section \ref{section-lightlike-spinor} we will study to the problem of finding  parallel spinors 
on the Lorentzian manifold in \eqref{globhypmetric}. For Riemannian manifolds, the corresponding Cauchy  problem was studied by Ammann, Moroianu and Moroianu  \cite{AmmannMoroianuMoroianu13} in relation to the Cauchy problem for Ricci-flat manifolds. But, in contrast to the Riemannian situation,  Lorentzian manifolds with parallel spinors are not necessarily Ricci-flat. Hence, for Lorentzian manifolds the Cauchy problem \ref{cp}(ii) for parallel null spinors  in general is {\em not} a special case  of the Cauchy problem for Lorentzian  Ricci-flat metrics   (which we review briefly in Section~\ref{section-curv}). However, since a parallel spinor $\phi$ on a Lorentzian manifold induces a parallel {\em Dirac current} $V_\phi$, in the case when $V_\phi$ is null, we can apply Theorem \ref{maintheo1} instead.
First we answer  question \ref{cc}(ii):  the parallel spinor $\phi$ induces a spinor field $\vf$ on the Cauchy hypersurface  $\M$, which satisfies the following {\em constraint conditions} 
\belabel{spinor-intro}\begin{array}{rcl}
 \nabla^S_X
\varphi &=& \tfrac{i}{2}\, \W(X)\cdot \varphi, \qquad\forall X\in T\M,
\\[1mm]
 U_{\varphi} \cdot \varphi &=& i\,u_{\varphi} \,\varphi, 
\end{array} \eeqn
in which $U_\vf$ is the {\em ``Riemannian'' 
Dirac current} of $\vf$ defined by
$\g(U_\vf,X)=-i( X\cdot \vf,\vf)$, $u_{\varphi}=
\sqrt{\g(U_{\varphi},U_{\varphi})} = \|\varphi\|^2$, and $\W$ is the Weingarten operator of $\M\subset (\bM,\bg)$. A spinor satisfying equations \eqref{spinor-intro} with a symmetric endomorphism field $\W$ is called {\em  generalised imaginary Killing spinor} or, more precisely, {\em imaginary $\W$-Killing spinor}.
We should mention that the case of {\em generalised real Killing spinors}, which correspond to parallel spinors for metrics of the form $\bg=dr^2+\g_r$, was studied by B\"ar, Gauduchon and Moroianu \cite{baer-gauduchon-moroianu05}.

In order to  answer  question \ref{cp}(ii) by applying  Theorem \ref{maintheo1}, one checks that the data $(W,U_\vf)$ associated to an imaginary $\W$-Killing spinor  on $(\M,\g)$ satisfy the constraint conditions \eqref{constraint-intro} for a parallel vector field. Then we can apply 
Theorem \ref{maintheo1}  and obtain:
\begin{mthm}\label{maintheo2}
Let $(\M,\g)$ be an analytic Riemannian spin manifold with an
analytic $\g$-symmetric endomorphism field $\W$ and $\varphi$ an
imaginary $\W$-Killing spinor on $(\M,\g)$.
Then $\big(\M,\g,\W,U_\vf\big)$ satisfies the constraint conditions \eqref{constraint-intro} and on the Lorentzian manifold 
$\left(\overline{\cU}, \bg= -\lambda^2 dt^2 +
\g_t\right)$
obtained in Theorem~\ref{maintheo1}, and 
 with parallel null vector field $V$, there  exists a parallel null spinor field $\phi$ with Dirac current $V$. The parallel spinor $\phi$  is obtained by parallel transport of $\vf$ along the lines $t\mapsto (t,x)$.
\end{mthm}
Finally, in Section \ref{examples} we give examples of {\em complete} Riemannian metrics satisfying the constraint equations \eqref{constraint-intro}, including a metric on the $2$-torus.

Our results in Theorems \ref{maintheo1} and \ref{maintheo2} are just the beginning of studying global hyperbolicity for manifolds with special holonomy and they suggest further questions:
\begin{itemize}
\item Can Theorems \ref{maintheo1} and \ref{maintheo2} be generalised to the smooth setting?
\item Under which conditions do exist long term solutions to the Cauchy problems \eqref{cp} that give a Lorentzian metrics on $\bM=\rr\times \M$?
\item Provided there is a long term solution on $\bM=\rr\times \M$,  under which conditions is the resulting Lorentzian manifold   globally hyperbolic?
\item Is there a classification of Riemannian manifolds saisfying the constraint conditions \eqref{constraint-intro} and \eqref{spinor-intro}?
\end{itemize}
Having the analogous situation for the Lorentzian Einstein equation in mind, for which the work of Choquet-Bruhat \cite{Foures-Bruhat52} settled  the problem in smooth case, it is very likely that the answer to the first question is positive. However, answers to these questions require  techniques that are beyond the scope of this paper and have to be postponed to future research.

\section{Metrics of the form $\bg = - \lambda^2 dt^2 +
\g_t$}\label{section-curv}

In the following, we consider product manifolds of the form $\bM:=\rr\times \M$,
where $\M$ is a smooth manifold of dimension $n$, with Lorentzian
metrics 
\be \label{Lmetric} \bg= -\lambda^2
dt^2 +\mathbf{g}_t, \ee
 where $\g_t$ is a $t$-dependent family of Riemannian
metrics on $\M$ and $\lambda=\lambda(t,x) $ is a smooth function on
$\bM$, the so-called {\em lapse function}.
For a metric of the form~\eqref{Lmetric}  we fix a time-like
unit vector field
\[T:=\lambda^{-1}\del_t.\]
In the following, by a bar we denote geometric objects defined by $\bg$ such as the
Levi-Civita connection $\bnab$. Objects without bar come from
$\g_t$ and do depend on the parameter $t\in\RR$. We will
indicate this with an (upper or lower) index $t$. Vector fields $U$ on
$\bM$ that are orthogonal to $T$ (or, equivalently, tangent to $\M$) can be considered in two ways: 
as sections of the bundle $\T^\perp \to \bM$, i.e., $U\in\Gamma(T^\perp)$,  or as $t$-dependent sections of the bundle $T\M\to\M$, i.e., $U_t\in \Gamma (T\M)$ for all $t$. Similarly, we will treat sections of tensor bundles of $T^\perp\to \bM$. Finally, when useful, we write functions on $\bM$ as $t$-dependent families of functions on $\M$, i.e., $u_t=u(t,.) $. Vector fields
(or functions) on $\M$ and their lifts to $\bM$ are denoted by the
same symbol.

The gradient of the lapse function is related to the derivative of $T$ as follows,
 \[ \bnab_T T = \grad^t(\log\lambda), \qquad \bnab_{\partial_t} T =
\grad^t(\lambda), \] where $\grad^t$ denotes the gradient with
respect to the metric $\g_t$. For $X,Y\in T\M$ denote by
\[
\mathrm{II}_t(X,Y):=-\bg( \bnab_XT,Y)
\]
the second fundamental form of $(\M,\g_t)\subset (\bM,\bg)$, i.e.,
we have
\[
\bnab_XY=\nabla^t_XY-\II_t(X,Y) T
\]
in which $\nabla^t$ denotes the Levi-Civita connection of $\g_t$.
The dual of the  second fundamental form is the {\em Weingarten operator} defined by
\[\II_t(X,Y)=\g_t(\W_t(X),Y),
\]
i.e., $\W_t=-\bnab T|_{T\M}$. The second fundamental form is
computed in terms of $\g_t$ as
\[
\II_t(X,Y)=-\tfrac{1}{2}\lambda^{-1} (\cal L_{\del_t}\g_t)(X,Y),
\]
where $\cal L$ denotes the Lie derivative. Hence, extending  $X$ and
$Y$ independent of $t$ we have
\be \label{2ff} \II_t(X,Y) &=& \!\!\! -\tfrac{1}{2}\lambda^{-1}
\del_t\g_t(X,Y)=:
-\tfrac{1}{2}\lambda^{-1} \dot{\g}_t(X,Y)\\
\dot{\II}_t(X,Y) &=&  \tfrac{1}{2}\lambda^{-1}\big(
\,\dot{(\log\lambda)} \dot{\g}_t(X,Y) - \ddot{\g}_t(X,Y)\,\big),
\label{2ff-1} \ee
where the dot denotes the partial $t$ derivative.
Moreover, 
for a symmetric $(2,0)$-tensor field $h$ and a one form $\mu$ we use
the notation
\begin{eqnarray*}  d^{\nabla}h(X,Y,Z)&:=&
(\nabla_Xh)(Y,Z)-(\nabla_Yh)(X,Z),\\
\mu\wedge h(X,Y,Z) &:=& \mu(X)h(Y,Z)-\mu(Y)h(X,Z).\end{eqnarray*}
The skew symmetric derivative $d^{\nabla}h$ satisfies the first
Bianchi identity
\begin{eqnarray}\label{Bianchi-dnabla}
d^{\nabla}h(X,Y,Z) + d^{\nabla}h(Y,Z,X) + d^{\nabla}h(Z,X,Y) = 0.
\end{eqnarray}
The trace of $d^{\nabla}h(X, \cdot, \cdot)$ is given by the
divergence and the trace of $h$, via
\begin{eqnarray}\label{Tr-dnabla}
\tr_{\g_t} d^{\nabla^t}h(X, \cdot, \cdot) = \div^t h \,(X) + d\,
\tr_{\g_t}h\, (X).
\end{eqnarray}
 The
curvature $\bR$ of $\bg$ defined as $\bR(U,V):= \left[
\bnab_X,\bnab_Y\right]-\bnab_{[U,V]}$
 is linked to the curvature $\R^t$ of $\g_t$ by the
 {\em Gau\ss\ equation}
 \be\label{gauss}
 \bR (X,Y,Z,U) \  = \ 
 \R_t (X,Y,Z,U)
 -\II_t(X,Z)\II_t(Y,U)+ \II_t(X,U)\II_t(Y,Z),
 \ee
with $X,Y,Z,U\in T\M$,
 the {\em Codazzi equation}
 \belabel{codazzi}
 \begin{array}{rcl}
 \bR (X,Y,Z,T) &=&
 d^{\nabla^t}\II_t(X,Y,Z)
 \\
 &=&
 -\frac{1}{2\lambda}
  \left(
 d^{\nabla^t}\dot\g_t(X,Y,Z)-(d\log \lambda)\wedge \dot\g_t)(X,Y,Z)\right),
\end{array}
\eeqn
 and the following formula, sometimes called {\em Mainardi
equation}, \be \label{mainardi} \bR(X,T,T,Y)& =&
\II_t(X,\W_t(Y))+\frac{1}{\lambda}\left(\dot\II_t(X,Y)+\Hess^t(\lambda)(X,Y)\right),
\ee
where $\Hess^t(f)=\nabla^t df$ denotes the Hessian of a function with respect to the metric $\g_t$,
and $\W_t$ is the Weingarten operator. Indeed,
since  $\bnab_TT=\grad^t(\log\lambda)$, we have
\begin{eqnarray*}
\bR(X,T,T,Y)
&=&
\Hess^t(\log\lambda)(X,Y)-\left( T(\bg(\bnab_XT,Y))- \bg(\bnab_{[T,X]}T,Y)\right) +
\bg(\bnab_XT,\bnab_T Y)
\\
&=&
\Hess^t(\log\lambda)(X,Y)-\cal L_T(\bg(\bnab T,.))(X,Y)+\bg(\bnab_XT,\bnab_YT ),
\end{eqnarray*}
which proves \eqref{mainardi}, when taking into account that
\beq
\cal L_T(\bg(\bnab T,.))(X,Y) &=&
-\lambda^{-1}\cal L_{\del_t}\II_t(X,Y)+X(\lambda^{-1})Y(\lambda)
\\
&=&
-\frac{\dot \II_t(X,Y)}{\lambda}-X(\log \lambda)Y(\log \lambda).
\eeq
and that
\[\Hess^t(\log\lambda)+d(\log\lambda)^2\ =\ \frac{1}{\lambda}\Hess^t(\lambda).\]
For the Ricci curvature
\[\bRic\ =\ -\bR(T,\cdot,\cdot,T)+\sum_{i=1}^n\bR(E_i,\cdot,\cdot,E_i)\]
equation (\ref{Tr-dnabla}) gives 
\begin{eqnarray}
\label{rictt}
\bRic(T,T) &=& \|
\II_t\|_{\g_t}^2
+\frac{1}{\lambda}\left( \tr^t(\dot \II_t)+\Delta_t(\lambda) \right)
\\
\label{rictx} 
\bRic(X,T)&=&d(\tr^t\II_t)(X)+\div^t
\II_t(X)
\ee
and
\belabel{ricxy}
\begin{array}{rcl}
\bRic(X,Y)&=&\Ric^t(X,Y) + \tr^t(\II_t)\II_t(X,Y) -
2\II_t(X,\W_t(Y))
\\
&&-\frac{1}{\lambda}\left( \dot\II_t(X,Y)
 +\Hess^t(\lambda)(X,Y)\right),
\end{array}\eeqn
in which $\tr^t$ is the trace, $\div^t$ the divergence and
$\Delta_t=\tr^t(\Hess^t)$ the Laplacian, all with respect to $\g_t$.
$\| \II \|^2_{\g_t}$ is the norm with respect to $\g_t$, which is
equal to $\tr^t(\W_t^2)$. Finally, for the scalar curvature we get
\beqn\label{scalar} \bscal = \scal^t +(\tr^t( \II_t))^2 -
3\|\II_t\|_{\g_t}^2 -\frac{2}{\lambda}\left( \tr^t(\dot
\II_t)+\Delta_t(\lambda) \right). \eeqn
These formulae give us the well known constraint and evolution
equations for Ricci flat Lorentzian metrics (see 
\cite{BartnikIsenberg04} for a review). In fact, the Lorentzian metric
\eqref{Lmetric} is Ricci flat if and only if, the Riemannian metrics
$\g_t$ together with the symmetric bilinear form $\II_t$ satisfy the
{\em constraint equations} 
\belabel{constr1}
\begin{array}{rcl} \scal^t
&=&\|\II_t\|_{\g_t}^2- \tr^t( \II_t)^2
\\
d\,\tr^t\II_t&=&-\div^t \II_t, \end{array}
\eeqn 
which follow from
setting \eqref{rictt}, \eqref{scalar} and \eqref{rictx} to zero, and
the {\em evolutions equation} for $\II_t$, \beqn \label{evol}
\dot\II_t(X,Y)= \lambda\Big( \Ric^t(X,Y) +\tr^t(\II_t) \II_t(X,Y) -
2 \II_t(X,\W_t(Y))\Big) -\Hess^t(\lambda)(X,Y), \eeqn which comes
from equation \eqref{ricxy} and in which the dot denotes the $t$-derivative. Rewriting this equation in terms of
$\g_t$ using \eqref{2ff}, it becomes an evolution equation for
$\g_t$, namely 
\beqn \label{evolg} 
\begin{array}{rcl}
\ddot{\g}_t(X,Y)&=&
\left(\dot{(\log\lambda)}-\frac{\tr^t(\dot
\g_t)}{2}\right)\dot{\g}_t(X,Y)-2\lambda \dot\g_t(X,\W_t(Y))
\\
&&{} +
2\lambda\Hess^t(\lambda)(X,Y)-2\lambda^2\Ric^t(X,Y).\end{array} \eeqn
On can check that the constraint equations \eqref{constr1} are preserved under this flow (see for example \cite[p. 438]{Koiso81}). For analytic data (initial conditions and $\lambda$)
one can apply the  Cauchy-Kowalevski Theorem (an excellent reference is
\cite{Folland95}) in order to obtain a unique analytic solution. That this can be done also for smooth data is the result of fundamental work by Y. Choquet-Bruhat in \cite{Foures-Bruhat52}.

\section{Constraint and evolution equations for parallel null vectors} \label{section-par-vf}

In this section we study the problem of extending a given Riemannian manifold to a Lorentzian manifold of the form \eqref{Lmetric} under the condition that the metric $\bg$  admits a  {\em parallel null vector field}. 

In the following
let $(\bM,\bg)$ be a time-oriented Lorentzian manifold of the form
\eqref{Lmetric} with time orientation  given by $T$ and let $V$ be a vector
 field
We denote the time component of $V$ by
\begin{eqnarray} \label{VTu}
u:=-\bg(T,V).\end{eqnarray}
If $V$ is null, i.e., $\g(V,V)=0$, then $u>0$. The vector fields 
 $T$ and $V$ define a global space-like
vector field $U$ on $\bM$ tangent to $\M$ by projecting $-V$ along
$T$ onto $\M$,
\beqn\label{VTU} U:=  uT-V,\ \quad \text{ i.e., } \quad V = uT-U,
\eeqn yielding $\bg(U,U)=u^2$.
On the other hand, if $U$ is a nowhere vanishing vector field on
$\bM$ which is tangent to $\M$ in any point, then $U$ is space-like,
$u:= \sqrt{\bg(U,U)}>0$ and $V := uT-U$ is null and future directed
with $u=-\bg(V,T)$.
Again we consider $U$ and $u$ as $t$-dependent families of vector fields $U_t=U(t,.)$ and functions $u_t=u(t,.)$ on $\M$.
Now we observe the following: 
\blem
\label{prlemma} Let $V$ be a
null vector field on $\bM$ and $\overline{X}\in T\bM$. Then
$\bnab_{\overline{X}}V=0$ if and only if
$\pr_{T^\perp}(\bnab_{\overline{X}}V)=0$. \elem \bprf Let $V=u T-U$ as
in \eqref{VTU}. Since $V$ is null we have
\[
0=\bg(\bnab_{\overline{X}}V,V)=u \bg(\bnab_{\overline{X}}V,T)- \bg(\bnab_{\overline{X}}V,U).
\]
Hence, under the assumptions that $\pr_{T^\perp}(\bnab_{\overline{X}}V)=0$ we have, in particular, that
$\bg(\bnab_{\overline{X}}V,U)=0$ which proves the claim.
\eprf
In the following, we will denote
 by $\dot{u}_t$ and $\dot{U}_t$  the
Lie derivatives \[
\dot{u}_t:= \del_t(u)\text{ and }\dot{U}_t :=
{\cal L}_{\partial_t} U_t = [\partial_t, U_t].\]
\bprop \label{nabVlemma} Let $V$ a future directed null vector field,
$U_t$ the corresponding space-like projection onto $\M$ as defined in
(\ref{VTU}) and $u_t = \sqrt{\bg(U_t,U_t)}=-\bg(T,V)$.
\begin{enumerate}
\item
Let $X \in T\M$. Then $\,\bnab_XV=0\,$ is equivalent to \be
\label{nabZ1} \nabla^t_XU_t&=&-u_t \W_t(X).\ee 
In particular, the condition
$\bnab_XV=0$ implies $du_t(X) = -\II_t(U_t,X)$.
\item $\bnab_{\del_t}V=0$ if and only if
\be\label{nabtZ} \bnab_{\del_t}U_t&=&\dot{u}_t \ T+u_t\ \grad^t
\lambda, \ee which is equivalent to \be\label{nabtZ1} \dot{U}_t &=&
 u_t \ \grad^t(\lambda)
+\lambda\W_t(U_t).\ee 
The condition $\,\bnab_{\del_t}V=0\,$ implies  $\, \dot{u}_t  =
d\lambda(U)$.
\end{enumerate}
\eprop
\bprf Using $V= u_t T-U$ and $\bnab_{\del_t}T=\grad^t(\lambda)$, we
obtain
\begin{eqnarray*} 0 \; = \;  \bnab_{\overline{X}} V &=& du(\overline{X}) T+u
\bnab_{\overline{X}} T -\bnab_{\overline{X}}U \\
&= & \left\{\begin{array}{ll} \big(du_t(X)+\II_t(X,U_t)\big) T-
u_t\W_t(X)-\nabla^t_XU_t,&\text{if $\overline{X}=X\in T\M$,}
\\[2mm]
\dot{u}_t T + u_t\ \grad^t(\lambda) -\bnab_{\del_t} U_t,&\text{if
$\overline{X}=\del_t$.}
\end{array}\right.
\end{eqnarray*}
Lemma \ref{prlemma} shows that the first equation is equivalent to
\eqref{nabZ1}. The second one is equivalent to \eqref{nabtZ}. Again,
applying Lemma \ref{prlemma} and
\begin{eqnarray} \label{nabtZa} \bnab_{\del_t}U=[\del_t,U] +\bnab_{U}\del_t = \dot{U}_t
+\bnab_{U}\lambda T = \dot{U}_t +d\lambda(U)T -\lambda
W_t(U_t)\end{eqnarray} we get the equivalence of \eqref{nabtZ} and
\eqref{nabtZ1}. With Lemma \ref{prlemma} equation
 \eqref{nabZ1} implies $\, du_t(X)=-\II_t(U_t,X)\,$ and \eqref{nabtZ} and
\eqref{nabtZa} imply $\,\dot{u}_t  = d\lambda(U).$ \eprf

Now we drop the assumption that $V$ is a {\em null} vector field for
a moment.
\blem\label{curvlemma} Let $V$ be a  vector field on $(\bM,\bg)$.
Then $V$ is parallel for $\bg$ if and only if
the following conditions are satisfied:
\be \bR (\del_t,X)V &=&0, \text{ for all $X\in T\M$}, \label{Rpde}
\\
\bnab_{\del_t}\bnab_{\del_t}V&=&0,
\label{Vpde}
\\
\bnab_{X}V|_{\{0\}\times \M} &=& 0, \text{ for all $X\in
T\M|_{\{0\}\times \M}$} \label{Vconstr}
\\
\bnab_{\del_t}V|_{\{0\}\times \M}&=&0. \label{Vic} \ee \elem
\bprf Clearly, if $V$ is parallel, all the conditions follow
immediately. Thus, let us assume the four conditions. Firstly, the
equation \eqref{Vpde} shows that the vector field $\bnab_{\del_t}V$
is parallel transported along the curves $t\mapsto(t,x)$. Hence,
because of the initial condition \eqref{Vic}, we get that
$\bnab_{\del_t}V=0$ everywhere. Using this, equation \eqref{Rpde}
gives that \beqn \label{crucial1} 0\ =\ \bR(\del_t,X)V\ =\
\bnab_{\del_t} \bnab_XV-\bnab_X\bnab_{\del_t} V-
\bnab_{[\del_t,X]}V\ =\ \bnab_{\del_t} \bnab_XV \eeqn whenever $X$
is the lift of a vector field of $\M$ to $\bM$, i.e., such that
$[\del_t,X]=0$. This shows that $\bnab_XV$ is parallel transported
along all $t\mapsto (t,x)$. Since we have assumed  that $\bnab_XV=0$
along the initial manifold $\M$, it also shows that $V$ is parallel
on $\bM$. \eprf

Now we will study the equations \eqref{Rpde} and \eqref{Vpde}
further.
\blem \label{curvVlemma} Let $(\bM,\bg)$ be a Lorentzian manifold
as in \eqref{Lmetric}.
\begin{enumerate}
\item There exists a vector field  $V\in \Gamma(T\bM)$
such that
\[\bR(X,Y)V=0\ \text{ for all $X,Y\in T\M$,}\]
 if and only if there is a smooth family of vector fields $\{U_t\}_{t\in \rr}$ and functions $\{u_t\}_{t\in \rr}$ on
 $\M$,
such that
\beqn \label{Vconstr2} \R^t(X,Y,Z,U_t)=u_t\,
d^{\nabla^t}\II_t(X,Y,Z)+\II_t(X,Z)\II_t(Y,U_t)-\II_t(Y,Z)\II_t(X,U_t),
\eeqn for all $X,Y,Z\in T\M$.
\item
There exists a vector field  $V\in \Gamma(T\bM)$ with
 \[
 \bR(T,X)V=0\ \text{ for all $X\in T\M$,}\]
  if and only if there is a smooth family of vector fields $\{U_t\}_{t\in \rr}$ and functions $\{u_t\}_{t\in \rr}$ on $\M$, such that
\beqn
\label{Vevolv2}
u_t \dot\II_t(X,Y) = \lambda ( d^{\nabla_t}\II_t)(U_t,Y,X)  -u_t \lambda \II_t(X,\W_t(Y)) -u_t\Hess^t(\lambda)(X,Y),
\eeqn
or equivalently,
\be
u_t\ddot \g_t(X,Y) &=&
{\lambda}^2  d^{\nabla^t}\Big(\frac{\dot \g_t}{\lambda}\Big)(U_t,Y,X)-
u_t\lambda \dot \g_t(X,\W_t(Y))
\nonumber
\\
&&{ }+
u_t\dot{(\log\lambda)}\dot{\g}_t(X,Y)
+2u_t \lambda\Hess^t(\lambda)(X,Y),
\label{Vevolvg}
\ee
for all $X,Y\in T\M$.
\end{enumerate}
\elem
\bprf We again use the relation between $V$ and $(U_t,u_t)$ given by
$V=u_tT-U_t$ with $u_t= - \bg(V,T)$, where $T=\lambda^{-1}\del_t $ is the
time like unit vector field.
 Then, for vectors $\overline{X}\in T\bM$ and $Y\in T\M$ we have
 \[\bR(\overline X,Y)V
 =
u_t\bR(  \overline X,Y)T-\bR(\overline X,Y)U_t.
\]
Hence, $\bR(\overline X,Y)V=0$ is equivalent to the equations
\be
\bR(\overline X,Y,U_t,T)&=&0\nonumber
\\
-u_t\bR(  \overline X,Y,Z,T)+\bR(\overline X,Y,Z,U_t)&=&0\label{rv}
\ee for all $Y,Z\in T\M$ and $\overline X\in T\bM$. Note that in
case $u_t$ has no zeros, the second equation implies the first as it
holds for all $Z\in T\M$ including $U_t$. Setting $\overline X=X\in
T\M$, the Gau\ss\ and Codazzi equations \eqref{gauss} and
\eqref{codazzi} show that \eqref{rv} is equivalent to
\eqref{Vconstr2}.

On the other hand for $\overline X=T$,
 the Codazzi and Mainardi equations \eqref{codazzi} and \eqref{mainardi} show that \eqref{rv} is equivalent to \eqref{Vevolv2}.
 Computing $\dot \II_t$ as
\[\dot \II_t(X,Y)= \frac{1}{2\lambda}\left(\dot{\log\lambda}\ \dot\g_t(X,Y)-\ddot{g}_t(X,Y)\right),\]
using \eqref{2ff}, shows that \eqref{Vevolv2},
 when written out in terms of $\dot\g_t$, just becomes \eqref{Vevolvg}.
 \eprf
 Next, we look at the second derivative of a null vector field in the
$t$-direction.
\blem\label{dtdtulemma} 
There exists a vector field $V$ on $\bM$
with
\[\bnab_{\del_t}\bnab_{\del_t}V=0,\]
if and only if there  exist a smooth family of vector fields
$U_t$ and functions $u_t$ on $\M$,
such that
\belabel{dtdtu1}\begin{array}{rcl}
 \ddot{U}_t &=& \lambda \left( \left[\del_t,\W_t(U_t)\right] +
\W_t(\dot{U}_t) -\lambda \W^2_t(U_t)\right) 
\\
&&{ }+ u_t\left( [\del_t,\grad^t\lambda ] -
\lambda\W_t(\grad^t\lambda)\right) + \dot{\lambda}_t \W_t(U)
+\left(2\dot u_t-d\lambda(U_t)\right) \grad^t\lambda
\end{array}
\eeqn
and
\be
\label{dtdtu2} \ddot{u}_t &=&\g_t( [\del_t,\grad^t\lambda ],U_t) +2
d\lambda (\dot{U}_t) -3\lambda d\lambda (\W_t(U_t))
-u_t\|\grad^t\lambda)\|^2_t. \ee
\elem
\bprf Note that $\bnab_{\del_t}T=\grad^t\lambda$ and
\[
\bnab_{\del_t}X=d\lambda(X)T+ \dot{X}_t  -\lambda W_t(X),\] for
$X\in \Gamma(T\M)$ but possibly depending on $t$. Using this, for
$V=u_t T-U_t$ we compute
%
\beq
\bnab_{\del_t}V&=& \dot u_t T+u_t\
\grad^t(\lambda) -\bnab_{\del_t} U_t \\
&=&  \big(\dot{u}_t-
d\lambda(U_t)\big)T + u_t \ \grad^t(\lambda) - \dot{U}_t  +\lambda
W_t(U_t).
\eeq
Applying $\bnab_{\del_t}$ to this,   we get
\beq \bnab_{\del_t}\bnab_{\del_t}V&=& \left( \ddot{u}_t
-\del_t(d\lambda(U_t)) +u_t \|\grad^t \lambda\|_t^2
-d\lambda(\dot{U}_t)+\lambda d\lambda (\W_t(U_t)) \right) T
\\
&& { }-\ddot{U}_t + \lambda \left( \left[\del_t,\W_t(U_t)\right] +
\W_t(\dot{U}_t) -\lambda \W^2_t(U_t)\right) \nonumber
\\
&&{ }+ u_t \left( [\del_t,\grad^t\lambda ] -
\lambda\W_t(\grad^t\lambda)\right) +\dot{\lambda}_t \W_t(U_t)
+\left(2\dot{u}_t - d\lambda(U_t) \right) \grad^t\lambda. \eeq
Setting $ \bnab_{\del_t}\bnab_{\del_t}V=0$ already implies
\eqref{dtdtu1}.
Computing
\beq \del_t(d\lambda(U_t)) &=&
\bg(\,\bnab_{\del_t}\overline{\grad}\lambda  , U\,) +
\bg(\,\overline{\grad}\lambda , \bnab_{\del_t}U \,) \\&=&
\g_t([\del_t,\grad^t\lambda],U_t) +d\lambda([\del_t,U_t]) -2\lambda
d\lambda (\W_t(U_t)), \eeq we obtain also \eqref{dtdtu2}. \eprf

For later purposes, we will now record the dualisation of  formula \eqref{dtdtu1}.

\blem\label{duallemma} Equation \eqref{dtdtu1} is equivalent to
\begin{equation}
\label{dtdtu1dual}
\begin{array}{rcl}
 \g_t(\ddot U_t,X) &=& -\frac{1}{2}\ddot \g_t(U_t,X)-\dot \g(\dot
U_t,X) -\frac{\lambda}{2}\dot \g_t(\W_t(U_t),X) 
\\&&{ }
+u_t\g_t( [\del_t,\grad^t\lambda],X)
+\frac{u_t}{2}\dot\g_t(\grad^t\lambda, X) +\left(2\dot
u_t-d\lambda(U_t)\right) d\lambda(X), \end{array}\eeqn
 for any
$X\in T\M$.\elem
\bprf We compute the first term in the right hand side of
\eqref{dtdtu1} as
\beq \lambda \g_t( \left[\del_t,\W_t(U_t)\right] ,X) &=& \lambda
\,\bg \big(\, \bnab_{\del_t} \W_t(U_t) -\bnab_{\W_t(U_t)}\del_t,X
\big)
\\ &=&
\lambda \,\del_t \left( \II_t(U_t,X)\right) -\lambda \,
\bg(\W_t(U_t),\bnab_{\del_t}X)+\lambda^2\II_t(\W_t(U_t),X)
\\&=&
-\frac{\lambda}{2}\del_t\left(\frac{1}{\lambda}
\dot\g_t(U_t,X)\right) +2\lambda^2\II_t(\W_t(U_t),X)
\\&=&
\frac{1}{2}\del_t(\log\lambda)\dot\g_t(U_t,X) -\frac{1}{2}
\left(\ddot\g_t(U_t,X)+\dot\g_t(\dot U_t,X) \right)
-\lambda\dot\g_t(\W_t(U_t),X). \eeq
Hence, all of \eqref{dtdtu1}
gives
\beq \g_t(\ddot U_t,X) &=&
\frac{1}{2}\del_t(\log\lambda)\dot\g_t(U_t,X) -\frac{1}{2}
\left(\ddot\g_t(U_t,X)+\dot\g_t(\dot U_t,X) \right)
-\lambda\dot\g_t(\W_t(U_t),X)
\\
&&
-\frac{1}{2}\dot\g_t(\dot U_t,X ) +\frac{\lambda}{2}\dot\g_t(\W_t(U_t)),X) -\frac{1}{2}\del_t(\log\lambda) \dot\g_t(U_t,X))
\\
&&{ }+ u_t\left(\g_t ([\del_t,\grad^t\lambda ] , X) +\frac{1}{2}\dot\g_t(\grad^t\lambda,X)\right)
 +\left(2\dot u_t-d\lambda(U_t)\right) d\lambda(X),
\eeq
which proves the Lemma.
\eprf
Using the previous Lemmas, we obtain
\btheo\label{Theorem-parallel-vf}
  Let
$(\M,\g, \W, U)$ be a Riemannian manifold together with a field of
$\g$-symmetric endomorphisms $\W$, with corresponding symmetric
bilinear form $\II:=\g(\W \cdot, \cdot)$, and a vector field $U$  satisfying the following constraint
equations \be\label{constraintU} \nabla^\g U+u\W&=&0,
\ee
where $u^2:=
\g(U,U)$.
Then, for any positive smooth function $\lambda$ on $\RR\times \M$,
a triple $(\g_t, U_t, u_t)$ of smooth one-parameter families of
Riemannian metrics, vector fields  and functions on $\M$ defines a
Lorentzian metric
\[
\bg=-\lambda^2dt^2 +\g_t
\]
on an open neighbourhood $\overline{\cU}(\{0\} \times \M) \subset
\rr\times \M$ with parallel null vector field
\[V=\frac{u_t}{\lambda} \del_t-U_t,\]
if and only if $\g_t$, $U_t$ and $u_t$ satisfy the following system
of PDEs on $\overline{\cU}$,
\be
 \ddot \g_t(X,Y)
 &=&\label{Vevolvg1sys}
 \frac{\lambda^2}{u_t}  d^{\nabla^t}\Big(\frac{\dot
\g_t}{\lambda}\Big)(U_t,Y,X)+\frac{1}{2}
 \dot \g_t(X,\dot\g_t^\sharp(Y))
+ \dot{(\log\lambda)}\dot{\g}_t(X,Y)
\\&&{}+
2\lambda\Hess^t(\lambda)(X,Y), 
\nonumber
\\
\g_t(\ddot U_t,X) 
\label{dtdtu1sys}
&=&
-\frac{\lambda^2}{2u_t}d^{\nabla^t}\Big(\frac{\dot
\g_t}{\lambda}\Big)(U_t,X,U_t) -
\dot\g_t(\dot U_t,X) -\frac{\dot{\log\lambda}}{2}\dot\g_t(U_t,X)
\\&&{} -\lambda\Hess^t(\lambda)(U_t,X)\nonumber
+u_t\g_t( [\del_t,\grad^t\lambda],X)
+\frac{u_t}{2}\dot\g_t(\grad^t\lambda, X) 
\\&&{ }+\left(2\dot
u_t-d\lambda(U_t)\right) d\lambda(X), \nonumber
\\[2mm]
\label{dtdtu2sys} \ddot{u}_t&=&\g_t( [\del_t,\grad^t\lambda ],U_t)
+2 d\lambda (\dot{U}_t) +\frac{3}{2} \dot\g_t (\grad^t(\lambda),
U_t) - u_t\|\grad^t\lambda\|^2_{\g_t}, \ee  with the initial
conditions 
\belabel{icdU0} 
\begin{array}{rcl} \g_0&=&\g,
\\
\dot \g_0&=& -2\lambda_0 \II,
\\
U_0&=&U,
\\
\dot U_0&=& u \ \grad^\g(\lambda_0)   +\lambda_0\W(U),
\\
u_0&=&u,
\\
\dot u_0&=&d\lambda_0(U). 
\end{array} \eeqn Here, $\dot\g_t^\sharp$
denotes the metric dual of $\dot\g_t$, i.e.,
$\g_t(X,\dot\g^\sharp_t(Y))=\dot\g_t(X,Y)$ and in terms of the
Weingarten operator $W_t$ of $(\M,\g_t)$ we have $\,\dot
\g_t(X,\dot\g_t^\sharp(Y)) = 4 \lambda^2 \g_t(W_t(X),W_t(Y)) =
-2\lambda \dot{\g}_t(\W_t(X),Y)$. \etheo
\bprf
Let $(\g_t,U_t,u_t)$ be a triple of one-parameter families of
Riemannian metrics, vector fields  and functions on $\M$.

First, assume that the Lorentzian metric $\bg$ defined by $\lambda$
and $\g_t$ admits a parallel null vector field $V$ defined by
$\lambda$, $U_t$ and $u_t$. Since $V$ is null, we have
$\g_t(U_t,U_t)=u_t^2>0$. Moreover, it implies $ \bR(T,X)V=0$ for all
$X\in T\M$ and hence, by \eqref{Vevolvg} in  \Lem{curvVlemma} and
$u_t>0$ we obtain equation \eqref{Vevolvg1sys}. Moreover, as $V$ is
parallel, \Lem{dtdtulemma} gives us  \eqref{dtdtu2} which is nothing
else than \eqref{dtdtu2sys}. Combining \Lem{dtdtulemma} and
\Lem{duallemma} shows that also \eqref{dtdtu1dual} holds. Now using
the obtained \eqref{Vevolvg1sys} in order to substitute  the term
$\ddot \g_t(U_t,X)$ in \eqref{dtdtu1dual}, we obtain
\eqref{dtdtu1sys}.

Now assume that, for a given $\lambda$, the triple $(\g_t,U_t,u_t)$
of Riemannian metrics, vector fields and functions on $\M$ solves
equations (\ref{Vevolvg1sys}, \ref{dtdtu1sys}, \ref{dtdtu2sys}) with
the given initial conditions. Using this solution, we define the
vector field  $V=u_t T-U_t$ on $\overline{\cU}$,  with
$T=\lambda^{-1}\del_t$. Then, by equation \eqref{Vevolvg} in Lemma
\ref{curvVlemma}, equation \eqref{Vevolvg1sys} shows that $V$
satisfies
\[\bR(T,X)V=0\]
for all $X\in T\M$.
We can use  equation \eqref{Vevolvg1sys} to
substitute $-\frac{1}{2}\ddot\g_t(U_t,X)$ into equation
\eqref{dtdtu1sys}. This equation becomes exactly equation
\eqref{dtdtu1dual}  in Lemma \ref{duallemma}.
 On the other hand,
equation \eqref{dtdtu2sys} is just \eqref{dtdtu2}, and hence, by
Lemma~\ref{dtdtulemma} we obtain that
\[\bnab_{\del_t}\bnab_{\del_t}V=0.\]
The initial condition $u_0=u$ together with the constraint
$\g(U,U)-u^2 = 0$ show that the vector field $V|_{\M }=u_0 T-U_0\in
T\overline{\cU}|_{\M}$ has constant length $0$  along the initial
manifold $\M \simeq \{0\}\times \M$. Then, as in
Lemma~\ref{prlemma}, we have
\[
0 = X \bg (V,V)|_{t=0} = 2\bg(\bnab_XV,V)|_{t=0} = 2u_0
\bg(\bnab_XV|_{t=0},T_0)|_{t=0} - 2\bg(\bnab_XV|_{t=0},U_0)
\]
which shows that $\,\bnab_XV|_{t=0}=0\,$ if
$\,\pr_{T^\perp}(\bnab_XV|_{t=0})=0$. But by  \Prop{nabVlemma} we have
that  $\pr_{T^\perp}( \bnab_XV|_{t=0})=0\,$, and hence
$\,\bnab_XV|_{t=0}=0\,$, if
\[\nabla^0_XU_0=-u_0\W_0(X),\]
which is the constraint condition \eqref{constraintU}. Thus, along
the initial manifold $\M$ we have
\[\bnab_XV|_{t=0}=0.\]
Finally, we show that, along the initial manifold $\M$ we also have $\bnab_{\del_t}V|_{t=0}=0$. As computed above, we have
\[
\bnab_{\del_t}V|_{t=0}\ =\ \big(\dot u_0-d\lambda(U_0)\big)T_0+u_0\
\grad^0(\lambda) -\dot U_0  +\lambda W_0(U_0) = 0,
\]
along $\M$,
because of the initial conditions  \eqref{icdU0}. Hence, all assumptions of   Lemma \ref{curvlemma}
 are satisfied and we obtain that $V$ is parallel on $\overline{\cU}$. But since $V$  has constant length $0$
 along the initial surface $\M$ and is parallel, it has constant length $0$   everywhere.
\eprf


\bbem\label{symrem}We observe that the Cauchy-Kowalevski theorem  can be applied to the PDE system  in Theorem \ref{Theorem-parallel-vf}, provided that all initial data are assumed as analytic.
It guarantees existence and
uniqueness of solutions to the evolution equations
\eqref{Vevolvg1sys} - \eqref{dtdtu2sys} for the given constraints
and initial conditions for data on $\M$  in an open neighbourhood of ${\{0\}\times
\M}$. However, we do not know
whether the  solution $\g_t$  defines a family of {\em symmetric} bilinear forms, and hence for small $t$ 
Riemannian metrics, on $\M$. 
The issue here is 
 that the right-hand-side of equation 
\eqref{Vevolvg1sys} in general does not map symmetric bilinear forms to symmetric bilinear forms, i.e., it  is not
an operator on  the bundle of  symmetric
bilinear forms on $\M$. 
In fact, the term $d^{\nabla^t}\!\big(\frac{\dot
\g_t}{\lambda}\big)(U_t,Y,X)$ {\em a priori} is not symmetric in
$(Y,X)$\footnote{We thank Olaf M\"uller for pointing out to us this gap
in an earlier draft of this paper.}.
Nevertheless, in Proposition \ref{Prop-example} we will give a class of
examples where the solution $\g_t$ is symmetric.

 Note that, due
to the Bianchi identity (\ref{Bianchi-dnabla}), the bilinear form
$d^{\nabla^t}\!\big(\frac{\dot \g_t}{\lambda}\big)(U_t, \cdot,
\cdot)$ is symmetric, if and only if the 2-form $\,
d^{\nabla^t}(\frac{\dot \g_t}{\lambda})(\cdot, \cdot ,U_t)$ vanishes
on $\M$. Written in terms of the Weingarten operator $W_t$ of
$(\M,\g_t)$ this is equivalent to the condition that the image of
the 2-form $d^{\nabla^t}\W_t$ is orthogonal to $U_t$ with respect to
$\g_t$. \ebem

\section{Solving the evolution equations  for analytic data}
\label{analytic}

For analytic metrics, we will now overcome the difficulty described in Remark \ref{symrem}. In fact, 
Lemma \ref{curvVlemma} reveals that the problematic evolution equation \eqref{Vevolvg1sys} for $\g_t$ is equivalent
to the curvature condition \eqref{Rpde} from Lemma \ref{curvlemma}. In order to overcome the symmetry issue, we next show that for
analytic Lorentzian manifolds of the form \eqref{Lmetric}, one can alternatively characterize parallel
null vector fields by relaxing \eqref{Rpde} and using a method similar to the one in \cite{AmmannMoroianuMoroianu13}.
This approach turns out to yield evolution equations for symmetric bilinear forms $\g_t$.

\blem\label{curvlemma2}
Let $(\bM,\bg)$ be an analytic Lorentzian manifold of the form \eqref{Lmetric}
and let $V$ be a analytic null vector field on $\bM$. Then $V$ is parallel for $\bg$, i.e.  $\bnab V =0$,
if and only if the following conditions are satisfied:
\be
\bR (X,V,V,Y)&=&0, \text{ for all $X,Y \in T\M$},
\label{Rpde2}
\\
\bnab_{\del_t}\bnab_{\del_t}V&=&0,
\label{Vpde2}
\\
\bnab_{X}V|_{\{0\} \times \M}&=& 0, \text{ for all $X\in
T\M|_{\{0\}\times \M}$ }, \label{Vconstr2b}
\\
\bnab_{\del_t}V|_{\{0\}\times \M}&=&0. \label{Vic2} \ee \elem
\bprf Clearly, if $V$ is parallel, all the conditions follow immediately.
On the other hand, assuming the four conditions, it follows from \eqref{Vpde2}
and \eqref{Vic2} as in the proof of Lemma \ref{curvlemma} that $\bnab_{\del_t}V=0$ everywhere.
It remains to show that $V$ is also parallel in spacelike directions. \\
For the rest of the proof let $X,Y$ be vectors in $T\M$ or lifts of vector fields on $\M$ to $\bM$.
We consider the sections $A,B$ of the bundle $H:=(T^\perp)^*\otimes T\bM\to \bM$ and the section $C$ of the bundle
$ K:=\Lambda^2(T^\perp)^* \otimes T\bM\to \bM$ defined by
\beq A(X) &:=& \bnab_X V, \\
     B(X) &:=& \bR(T,X)V, \\
     C(X,Y)&:=& \bR(X,Y)V,
\eeq for $X,Y \in T\M$ and denote also by  $\bnab$  the covariant derivatives on
$H$ and $K$ induced by the Levi-Civita connection $\bnab$  of $\bg$. In order to verify that $A \equiv 0$, we decompose as in
\eqref{VTU}, $V = uT- U$, denote by $N$ the unit length space-like
vector field $N:= \frac{U}{u} \in\Gamma (T^\perp)$  and show that
the tripel $(A, B, C)$ solves the linear PDE

\be
\bnab_{\partial_t} {\begin{pmatrix} \alpha \\ \beta \\ \gamma \end{pmatrix}} = Q \begin{pmatrix} \alpha \\ \beta \\
\gamma \end{pmatrix}, \label{nabla-eq}
\ee
where $\begin{pmatrix} \alpha \\ \beta \\ \gamma \end{pmatrix} \in \Gamma(H \oplus H \oplus K)$ and $Q$ is
the linear operator on $H \oplus H \oplus K$ given by
\[ \begin{pmatrix} \alpha \\ \beta \\ \gamma \end{pmatrix} \stackrel{ Q}{\longmapsto} \lambda \left(%
\begin{array}{c}
 \beta + \alpha \circ \W \\[0.1cm]
  d^{\bnab}\beta(N,\cdot) + \gamma(W(N),\cdot) + \gamma(N,W(\cdot)) + (\bR(T,\cdot) \wedge \alpha)(N, \cdot) + \gamma(\grad(\log\lambda),\cdot) \\
  d^{\bnab}\beta + \gamma(W(\cdot),\cdot) + \gamma(\cdot, W(\cdot)) + \bR(T,\cdot)\wedge \alpha
\end{array}%
\right).
\]
From $N=T - \frac{1}{u}V$ and $\bnab_{\partial_t}V = 0$ it
follows that \be \bnab_T N = \grad(\log\lambda) +
\frac{\dot{u}}{\lambda u^2} V. \label{nabN2} \ee
Furthermore,  for $(t,x) \in \{t\}\times \M$ let  
$(e_1,...,e_n)$ denote a $\g_t$-orthonormal basis for $T_x\M$. We
observe that (\ref{Rpde2}) combined with the skew-symmetries of
$\bR$ yields the following identities (at $(t,x)$): \be \bR
(X,V,V,T)&=& \bR(X,V,V,N)\  =\  0, \nonumber
\\
\bR (T,X,V,T)&=& \bR(T,X,V,N) \ =\  \bR(N,X,V,N) \ =\  \bR(N,X,V,T), \nonumber
\\
\bR (T,X)V&=& \sum_{i=1}^n \bR(T,X,V,e_i)e_i - \bR(T,X,V,T)T\nonumber
\\
& =& \sum_{i=1}^n \bR(N,X,V,e_i)e_i - \bR(N,X,V,T)T \nonumber \\
&=& \bR(N,X)V,
\label{symR1}
\\
\bR(T,V)V &=& \sum_{i=1}^n \bR(N,V,V,e_i)e_i - \bR(N,V,V,N)T \ =\  0,
\\
\bR(X,V)V &=& 0.
\ee
In the following, we view $\bR$ as 2-form on $\bM$ taking values in the $\bg$-skew-symmetric endomorphisms.
With these preparations at hand and using the second Bianchi identity for $\bR$ we compute
\beq
(\bnab_{T}A)(X) &=& \bnab_T (A(X)) - A(\bnab_T X)
\\& = &\bR(T,X)V + \bnab_{[T,X]}V - \bnab_{\bnab_TX}V \\
&=&B(X) + A(W(X)),
\eeq
as well as
\beq
(\bnab_{T}B)(X) &=& \bnab_T (B(X)) - B(\bnab_T X) \\
                &=& \bnab_T (\bR(T,X)V) - \bR(T,\bnab_TX)V \\
                &\stackrel{{\eqref{symR1}}}{=}& \bnab_T(\bR(N,X)V) - \bR(N,\prT \bnab_TX)V \\
                &=& \bnab_T (\bR)(N,X)V + \bR(\bnab_TN,X)V + \bR(N,\bnab_TX)V - \bR(N,\prT \bnab_TX)V \\
                &\stackrel{{\eqref{symR1}}}{=}& \bnab_N (\bR)(T,X)V + \bnab_X (\bR)(N,T)V + \bR(\bnab_TN,X)V \\
                &\stackrel{(\ref{nabN2})}{=}& \bnab_N(\bR(T,X)V) - \bR(\bnab_NT,X)V - \bR(T,\bnab_NX)V - \bR(T,X)\bnab_NV \\
                && + \bnab_X(\bR(N,T)V) - \bR(\bnab_XN,T)V - \bR(N,\bnab_XT)V - \bR(N,T)\bnab_XV \\
                &&
                +\bR(\grad(\log\lambda),X)V \\
                &=& (\bnab_N B)(X) + C(W(N),X) -\bR(T,X)A(N) \\
                &&-(\bnab_X B)(N) -C(W(X),N) +\bR(T,N) A(X) + C(\grad(\log\lambda),X),
\eeq
and 
\beq
(\bnab_{T}C)(X,Y) &=& \bnab_T (\bR(X,Y)V) - \bR(\bnab_T X,Y)V - \bR(X,\bnab_TY)V \\
                  &=& (\bnab_T \bR)(X,Y)V = (\bnab_X \bR)(T,Y)V + (\bnab_Y \bR)(X,T)V \\
                  &=& \bnab_X(\bR(T,Y)V) - \bR(\bnab_XT,Y)V - \bR(T,\bnab_XY)V - \bR(T,Y) \bnab_XV \\
                  && + (\bnab_Y \bR(X,T)V ) - \bR(\bnab_YX,T)V - \bR(X,\bnab_YT)V - \bR(X,T) \bnab_YV \\
                  &=& (\bnab_X B)(Y) + C(W(X),Y) - \bR(T,Y)A(X) - (\bnab_Y B)(X) \\
                  && -C(W(Y),X) + \bR(T,X)A(Y).
\eeq These calculations prove that $(A,B,C)$ satisfy equation \eqref{nabla-eq}, being a linear PDE
which separates into $\partial_t = \lambda \cdot T$-derivatives on
the left-hand-side and spacial derivatives of order at most one on
the right-hand-side. At $t=0$ we have that $A=0$ by assumption.
Differentiating this again in direction of $\M$ and
skew-symmetrizing yields that also $C=0$ at $t=0$. Finally, it
follows that at $t=0$ we have $B(X)=\bR(T,X)V
= C(N,X) = 0$ by \eqref{symR1}. As all data are analytic,
the Cauchy-Kowalevski Theorem guarantees the existence of a unique
analytic solution. From the initial conditions it follows that $A
\equiv 0$, $B \equiv 0$, $C \equiv 0$, i.e. $\bnab_X V = 0$
everywhere. \eprf
\bbem \label{rewrite}
By inserting $V=u(T-N)$, condition (\ref{Rpde2}) from Lemma \ref{curvlemma2} becomes equivalent to
\be
\bR(X,T,T,Y) = \bR(N,X,Y,T) + \bR(N,Y,X,T) - \bR(X,N,N,Y)\label{R-bed}
\ee
for $X,Y \in T\M$. Rewriting (\ref{R-bed}) in terms of $t-$dependent data on $\M$ using (\ref{2ff})-(\ref{mainardi}),
solving for the $\ddot \g_t$-term and setting $U_t = u_t N$  yields the equivalent formulation

\be
 \ddot \g_t(X,Y)
 &=& \frac{\lambda^2}{u_t}
 \left( d^{\nabla^t}\Big(\frac{{\dot
\g_t}}{\lambda}\Big)(U_t,X,Y)+ d^{\nabla^t}\Big(\frac{{\dot
\g_t}}{\lambda}\Big)(U_t,Y,X) \right) + \frac{1}{2} \dot
\g_t(X,\dot\g_t^\sharp(Y))
 \nonumber \\
 &&  + \dot{(\log\lambda)}\dot{\g}_t(X,Y)  + 2 \lambda \Hess^t(\lambda)(X,Y) + 2 \frac{\lambda^2}{u_t^2} \R_t(X,U_t,U_t,Y)
 \nonumber \\
&& +  \frac{1}{2u_t^2} \Big(\dot \g_t(X,Y) \dot \g_t(U_t,U_t) - \dot
\g_t(X,U_t) \dot \g_t(Y,U_t)\Big).   \nonumber \ee
\ebem 
The previous calculations
directly imply the following statement, which in contrast to Theorem
\ref{Theorem-parallel-vf} we can prove for analytic data only:
\btheo\label{Theorem-parallel-vf2}
  Let
$(\M,\g, \W, U)$ be an analytic Riemannian manifold together with a
field of $\g$-symmetric and analytic endomorphisms $\W$, with
corresponding symmetric bilinear form $\II:=\g(\W.,.)$, and  an analytic
vector field $U$ 
satisfying the following constraint equation
\be\label{constraintU2} \nabla^\g U+u\W&=&0,
\ee
where $u^2=
\g(U,U)$.
Then,  for any positive analytic function $\lambda$  on $\RR\times \M$, the triple $(\g_t,
U_t, u_t)$ of analytic one-parameter families of Riemannian metrics,
vector fields and functions on $\M$ defines an analytic Lorentzian
metric
\[
\bg=-\lambda^2dt^2 +\g_t
\]
on an open neighbourhood $\overline{\cU}(\{0\} \times \M) \subset
\rr\times \M$ with analytic parallel vector field
\[V=\frac{u_t}{\lambda} \del_t-U_t,\]
if and only if $\g_t$, $U_t$ and $u_t$ satisfy the following system
of PDEs on $\overline{\cU}$,
\be
 \ddot \g_t(X,Y)
 &=& \frac{\lambda^2}{u_t}
 \left( d^{\nabla^t}\Big(\frac{{\dot
\g_t}}{\lambda}\Big)(U_t,X,Y)+ d^{\nabla^t}\Big(\frac{{\dot
\g_t}}{\lambda}\Big)(U_t,Y,X) \right) + \frac{1}{2} \dot
\g_t(X,\dot\g_t^\sharp(Y))
 \nonumber \\
 &&  + \dot{(\log\lambda)}\dot{\g}_t(X,Y)  + 2 \lambda \Hess^t(\lambda)(X,Y) + 2 \frac{\lambda^2}{u_t^2} \R_t(X,U_t,U_t,Y)
 \label{Vevolvg1sys2} \\
&& +  \frac{1}{2u_t^2} \Big(\dot \g_t(X,Y) \dot \g_t(U_t,U_t) - \dot
\g_t(X,U_t) \dot \g_t(Y,U_t)\Big). \nonumber\\
\g_t(\ddot U_t,X) &=&
\nonumber
- \frac{\lambda^2}{2u_t}  d^{\nabla^t}(\frac{\dot
\g_t}{\lambda})(U_t,X,U_t)
-\frac{1}{2} \dot{(\log\lambda)}\dot{\g}_t(U_t,X) - \lambda \Hess^t(\lambda)(U_t,X) \\
&&   \label{dtdtu1sys2} { }- \dot \g_t (\dot U_t,X) 
+u_t\g_t( [\del_t,\grad^t\lambda],X)  +\frac{u}{2}\dot\g_t(\grad^t\lambda, X)
\\
&&{ } +\left(2\dot u_t-d\lambda(U_t)\right) d\lambda(X), \nonumber
\\[2mm]
\label{dtdtu2sys2} \ddot{u}_t&=&\g_t( [\del_t,\grad^t\lambda ],U_t)
+2 d\lambda (\dot{U}_t) +\frac{3}{2} \dot\g_t (\grad^t(\lambda),
U_t) - u_t\|\grad^t\lambda\|^2_{\g_t}, \ee  with the initial
conditions 
\belabel{icg02}
\begin{array}{rcl}
 \g_0&=&\g,
\\
\dot \g_0&=& -2\lambda_0 \II,
\\
U_0&=&U,
\\
\dot U_0&=& u \ \grad^\g(\lambda_0)   +\lambda_0\W(U),
\\
u_0&=&u,
\\
\dot u_0&=&d\lambda_0(U).
\end{array}
\eeqn

\etheo
\bprf
This is in complete analogy to the proof of Theorem~\ref{Theorem-parallel-vf}: one verifies as in Theorem~\ref{Theorem-parallel-vf} and in this case additionally using Remark \ref{rewrite}
that the equations  \eqref{constraintU2} -  \eqref{dtdtu2sys2} and the initial conditions \eqref{icg02} are just a reformulation of the conditions
\eqref{Rpde2}-\eqref{Vic2} appearing in Lemma \ref{curvlemma2} in terms of $t$-dependent data on $\M$. Note that in analogy
to Theorem  \ref{Theorem-parallel-vf} the evolution equation \eqref{dtdtu1sys2} arises from substituting the
term $\ddot \g_t(U_t,X)$ in \eqref{dtdtu1dual} via \eqref{Vevolvg1sys2} and here additionally using that
$((d \log \lambda) \wedge \dot \g_t) (U_t,U_t,X)=0$.
\eprf

\bbem In Theorems \ref{Theorem-parallel-vf} and
\ref{Theorem-parallel-vf2} the constraint equation
\eqref{constraintU} (or \eqref{constraintU2}, respectively) is
needed in order to ensure that $V$ is parallel along $\M$. Note also
that the constraint $u^2-\g(U,U)=0$ is compatible with the initial
conditions \eqref{icdU0} (resp. \eqref{icg02}). Indeed,
\[
2u \dot u_0
=
\dot{u^2}|_{t=0}
=
2 \g_0(\bnab_{\del_t}U_t|_{t=0},U_0)
=
2 \g_0(\dot U_0,U_0) -2\lambda \II_0(U_0,U_0)
=
2ud\lambda(U).
\]
\ebem
In contrast to \eqref{Vevolvg1sys}, the $\g_t$-evolution equation \eqref{Vevolvg1sys2}
is manifestly an equation in the bundle of symmetric bilinear forms on $\M$, i.e. at
least for small $t$ the solutions $\g_t$ are Riemannian metrics on $\M$. By the Cauchy-Kowalevski Theorem we obtain the following corollary which gives the statement of Theorem \ref{maintheo1} in the introduction:

\bcorl{Cor-KWparallel}
 Let
$(\M,\g, \W, U)$ be an analytic Riemannian manifold together with
a field of $\g$-symmetric, analytic endomorphisms $\W$, with
corresponding symmetric bilinear form $\II:=\g(\W.,.)$,  and an analytic
vector field $U$ satisfying the following constraint equation
\be\label{constraintU-2} \nabla^\g U+u\W&=&0,
\ee
where $u^2=
\g(U,U)$. 
Then, for any positive analytic
 function  $\lambda$ on $\RR\times \M$ there exists an open neighbourhood $\overline{\cU}(\{0\}\times \M) \subset
 \RR \times \M$  and an unique analytic Lorentzian metric
\[ \bg=-\lambda^2dt^2 +\g_t\]
on $\overline{\cU}$ which admits an analytic null
 parallel  vector field  $\;V=\frac{u_t}{\lambda} \del_t-U_t$, where
 $(\g_t, U_t, u_t)$ are solutions of the evolution equations \eqref{Vevolvg1sys2},
\eqref{dtdtu1sys2}, \eqref{dtdtu2sys2} with initial conditions
\eqref{icg02}.
 \ecor

\medskip

Finally, we consider an explicit example where  we find a solution for the evolutions equations in both Theorems \ref{Theorem-parallel-vf} and \ref{Theorem-parallel-vf2}.

\begin{Proposition}\label{Prop-example}
Let $(M,\g,\W,U,u)$ be a Riemannian manifold with an symmetric
endomorphism field $\W$, a vector field $U$ and a function $u$ on
$M$ satisfying the constraint equations
\begin{eqnarray*} \nabla^{\g} U &=& - u \W,\\
                  \g(U,U) &=& u^2 > 0.
                  \end {eqnarray*}
Let, in addition, $\W$ be a Codazzi tensor, i.e.,
$d^{\nabla^{\g}}\W=0$, and $\lambda=1$. Then
\begin{eqnarray}
\g_t &:=& \g - 2t \g(W(\cdot),\cdot) + t^2\g(W^2(\cdot),\cdot) =
\g((1-tW)^2(\cdot),\cdot),  \label{special gt}\\
U(t,x) &:=& \frac{1}{(1-t\W_x)} U(x) = \sum\limits_{k=0}^{\infty} \W_x^k(U(x)) t^k, \label{special Ut} \\
u(t,x) &:=& u(x). \label{special ut}
\end{eqnarray}
are solutions
to the evolution equations in both Theorem \ref{Theorem-parallel-vf} and Theorem \ref{Theorem-parallel-vf2}. These solutions are defiend on 
\[ \overline{\cU}(\{0\}\times \M) := \{ (t,x) \in \RR \times \M \mid
t\|\W_x\|_{\g_x} < 1 \}.\] 
In particular, the above solution $\g_t$ to the evolution equation \eqref{Vevolvg1sys} is a symmetric bilinear form.
\end{Proposition}

\begin{proof}
For $\lambda=1$, the evolution equations of Theorem
\ref{Theorem-parallel-vf} reduce to
\begin{eqnarray}
\ddot{\g}_t(X,Y)&=& \tfrac{1}{u_t}(d^{\nabla^t}\dot{\g}_t)(U_t,Y,X)
-
\dot{\g}_t(X,\W_t(Y)), \label{gt special2}\\
\g(\ddot{U}_t,X) &=& - \tfrac{1}{2u_t}
(d^{\nabla^t}\dot{\g}_t)(U_t,X,Y) - \dot{\g}_t(\dot{U}_t,X), \label{Ut special2}\\
\ddot{u}_t &=& 0, \nonumber
\end{eqnarray}
whereas the evolution equations of Theorem \ref{Theorem-parallel-vf2} reduce to
\begin{eqnarray}
\ddot \g_t(X,Y)
 &=&  \tfrac{1}{u_t} \left( (d^{\nabla^t}{\dot
\g_t})(U_t,X,Y) + (d^{\nabla^t}{\dot \g_t})(U_t,Y,X) \right) - \dot \g_t(X,W_t(Y))    \nonumber \\
&& + \frac{2}{u_t^2} \R_t(X,U_t,U_t,Y) + \frac{1}{2u_t^2} \big(\dot
\g_t(X,Y) \dot \g_t(U_t,U_t)
 - \dot \g_t(X,U_t) \dot \g_t(Y,U_t)\big),   \label{gt special}\\
\g(\ddot{U}_t,X) &=& - \tfrac{1}{2u_t}
(d^{\nabla^t}\dot{\g}_t)(U_t,X,U_t) - \dot{\g}_t(\dot{U}_t,X) \label{Ut special}\\
\ddot{u}_t &=& 0, \nonumber
\end{eqnarray}
with initial conditions $\g_0=\g$, $\dot{\g}_0=-2\II$, $U_0=U$,
$\dot{U}_0=\W(U)$, $u_0=u$, $\dot{u}_0=0$ in both cases. Hence $u(t,x)=u(x)$. It remains to show, that $\g_t$ and $U_t$ given by
(\ref{special gt}) and (\ref{special Ut}) satisfy (\ref{gt special})
and (\ref{Ut special}) as well as (\ref{gt special2})
and (\ref{Ut special2}) in case that $\W$ is a Codazzi tensor.
To this end, we observe that the definitions (\ref{special gt}) and (\ref{special Ut}) imply
\belabel{def gt-1}
\begin{array}{rcl}
\dot{\g}_t(X,Y) &=& - 2 \g(\W(1-t\W)(X),Y))
 \ = \ -2 \g_t\big(\tfrac{W}{(1-t\W)}(X),Y\big)
 \\[2mm]
 \ddot{\g}_t(X,Y) & =&  + 2\g(\W^2(X),Y)), \\[2mm]
\dot{U}_t&=& \tfrac{\W}{(1-t\W)^2}U, \\[2mm]
\ddot{U}_t& = &
\tfrac{2\W^2}{(1-t\W)^3}U. 
\end{array}\eeqn
Hence, for  the Weingarten operator $\W_t$ of $(\M,\g_t)$ we obtain
\begin{eqnarray} \W_t = \frac{\W}{(1-t\W)}. \label{Wt special}
\end{eqnarray}
Next we show, that $\W_t$ is a Codazzi tensor for $\g_t$. Since
$-2\II_t = \dot{\g}_t$, this is equivalent to
$\,d^{\nabla^t}{\dot{\g}}_t= 0\,$. Since $\W$ is a Codazzi tensor
for $\g$ by definition, the tensor field $\B_t:=(1-t\W)$ is Codazzi
tensor for $\g$ as well. Therefore, the Levi-Civita connection of
\[\g_t := \B_t^*\g = \g(\B_t(\cdot),\B_t( \cdot)) =
\g(\B_t^2(\cdot),\cdot)\] is given by
\be
\nabla^t= \B_t^{-1} \circ
\nabla^g \circ \B_t. \label{lc-sp}
\ee
It follows
\begin{eqnarray*}
(d^{\nabla^t}\dot{\g}_t)(X,Y,Z) 
&=& -2 X\big(\g(\W(Y),\B_t(Z))\big) +2 \g( B_t\W(\nabla^t_XY),Z) +2
\g(\B_t\W(Y), \nabla^t_XZ) \\
& &  + 2 Y\big(\g(\W(X),\B_t(Z))\big) - 2 \g( B_t\W(\nabla^t_YX),Z)
 - 2 \g(\B_t\W(X), \nabla^t_YZ)\\
 &=& -2 \g(\nabla^{\g}_X(\W(Y)), \B_t(Z)) - 2 \g(\W(Y),
 \nabla^{\g}_X(\B_t(Z)))\\
 && + 2 \g(\nabla^{\g}_Y(\W(X)), \B_t(Z)) +  2 \g(\W(X),
 \nabla^{\g}_Y(\B_t(Z)))\\
 && + 2 \g(\W  (\nabla^{\g}_X (\B_t(Y))),Z)  + 2 \g(\W(Y),\nabla^{\g}_X(\B_t(Z))) \\
 && -  2 \g(\W  (\nabla^{\g}_Y (\B_t(X))),Z) -  2
 \g(\W(X),\nabla^{\g}_Y(\B_t(Z)))\\
 &=& - 2 \g(d^{\nabla^{\g}}\W(X,Y), \B_t(Z)) - 2\g(W([X,Y],\B_t(Z))\\
 &&  +  2  \g(d^{\nabla^{\g}}\B_t(X,Y), \W(Z)) +  2
 \g(\B_t([X,Y],\W(Z))\\
 &=& 0.
\end{eqnarray*}
Finally, we compute the curvature term appearing in (\ref{gt special}). It follows from (\ref{Wt special}) and (\ref{lc-sp}) that
\be
\nabla^t U_t &=& -u W_t, \nonumber \\
\g_t(U_t,U_t) &=& u^2,
\ee
and consequently, as $d^{\nabla^t}W_t = 0$, we obtain
\be
\R_t(X,U_t,U_t,Y) &=& \g_t (\nabla^t_X (-u W_t(U_t)) - \nabla^t_{U_t} (-u W_t(X)) - \nabla^t_{[X,U_t]}U_t,Y) \nonumber \\
&{=} & -X(\sqrt{\g_t(U_t,U_t)}) \g_t(W_t(U_t),Y) + U_t(\sqrt{\g_t(U_t,U_t)}) \g_t(W_t(X),Y)  \nonumber\\
& = & \g_t(W_t(X),U_t) \g_t(W_t(Y),U_t) -\g_t(W_t(X),Y) \g_t(W(U_t),U_t) \nonumber \\
& = & \frac{1}{4} \left( \dot \g_t (X,U_t) \dot \g_t(Y,U_t) - \dot \g_t(X,Y) \dot \g_t (U_t,U_t) \right).
\ee
With this property and using equation (\ref{Wt special}), the evolution equations (\ref{gt
special2}) and (\ref{Ut special2}) as well as (\ref{gt
special}) and (\ref{Ut special}) reduce further to the same system, namely
\belabel{gt special-2}
\begin{array}{rcl}
\ddot{\g}_t(X,Y)&=& - \dot{\g}_t(X,\W_t(Y))\  =\ - \dot{\g}_t(X,\tfrac{\W}{(1-t\W)}(Y)), \\
\g_t(\ddot{U}_t,X) &=& - \dot{\g}_t(\dot{U}_t,X). \end{array}\eeqn
The system (\ref{gt special-2})   obviously has the solution given in
equations (\ref{def gt-1}). 
\end{proof}

\bbem Many of the previous statements admit a more general
formulation for parallel causal vector fields $V$ of constant
length, i.e. $\bnab V = 0$ and $\bg(V,V) \equiv c \leq 0$. However,
in case of a timelike parallel vector field on $\bM$ one always has
a local metric splitting of $\bM$ into a line and a Riemannian
factor. For instance, if we replace the constraint equation $\g(U,U)
- u^2 = 0$ in Theorem \ref{Theorem-parallel-vf} by $\g(U,U)-u^2 = c
= -1$, then obviously for $\lambda \equiv 1$ the system of equations has
the trivial solution $\g_t\equiv \g$, $U_t\equiv 0$ and $u_t\equiv
1$ which gives the parallel timelike vector field $\del_t$ on the
product metric $\bg=-dt^2+\g$.
\ebem

\section{Constraint and evolution equations for parallel null
spinors}\label{section-lightlike-spinor}

In this section we  assume in addition, that $(\bM,\bg)$ is a Lorentzian {\em
spin} manifold. For a spinor field $\phi$ on $(\bM,\bg)$ we define
its Dirac current $V_{\phi} \in \fX(\bM)$ by
\[ \bg(V_{\phi},   X) = - \langle  X \cdot \phi, \phi \rangle\,,  \qquad  \forall\ X \in \fX(\bM).\]
The vector field $V_{\phi}$ is future oriented, causal, i.e., 
$\bg(V_{\phi},V_{\phi}) \leq 0$ and the zero sets of $V_{\phi}$ and
$\phi$ coincide. If $\phi$ is parallel, $V_{\phi}$ is parallel as
well, and thus either null or time-like. We call a spinor field $\phi$
{\em null}, if its Dirac current $V_{\phi}$ is null. In this case we
have $V_{\phi} \cdot \phi=0$ and $\langle \phi,\phi \rangle = 0$.

From now on, we assume that $(\bM,\bg)$ admits a 
a {\em  parallel null} spinor field $\phi$. Then, for  its Dirac
current $V:=V_{\phi}$ we apply the notations and results of Section
\ref{section-par-vf}. We fix a time-orientation $T$, and consider as
in (\ref{VTu}) and (\ref{VTU}) the projection $U$ of $-V$ onto
$T\M$, the function $u:= - \bg(V,T)$ and the unit vector field
$N:= \tfrac{1}{u}U$. Since $\phi$ is parallel, the Ricci endomorphism is zero or $2$-step nilpotent, i.e.,
$\bRic^2=0$. This is equivalent to 
$\bRic = f \cdot (V^\flat)^2$ with  a function $f$ on
$\bM$. This implies
\beq 
\bRic(T,T)&=&f\, u^2\\
\bRic(X,T)&=&f\, u^2\g(N,X)\\
\bRic(X,Y)&=&f\, u^2\g(N,X)\g(N,Y)
\eeq
for $X,Y\in T\M$.
Therefore,
\begin{eqnarray*} \bRic(T,T) &=& \bRic(N,N) \;=\; \bRic(N,T) \; =\;  fu^2, \\
\bRic(X,Y) &=& 0 \quad \mbox{ if  $X \in \mathrm{span}(T,N)^{\bot}$
or $Y \in \mathrm{span}(T,N)^{\bot}$}.
\end{eqnarray*}
In particular, the scalar curvature $\bscal$ of $(\bM,\bg)$ vanishes.
If $(\M,\g)$ is a space-like hypersurface of $(\bM,\bg)$ with the
normal vector field $T$, the second fundamental form $\II$, and the
Weingarten operator $\W$, the Ricci-tensor and the scalar curvature
of $(\bM,\bg)$ and $(\M,\g)$ are related by
\begin{eqnarray*} \bRic(T,T) &=& \bRic(N,T) \;  = \;
\bRic(N,N) \; = \; \tr_{\g} d^{\nabla}\II(N)  \label{Ricci-1}\\
\bRic(X,T)&=& \tr_{\g} d^{\nabla}\II(X)  \label{Ricci-2}\\
\bRic(X,Y)&=& \Ric(X,Y) - d^{\nabla}\II(N,X,Y) - \II(X,\W(Y)) +
\tr_{\g}\II \cdot\II(X,Y) \label{Ricci-3} \\[0.1cm]
\bscal &=& \scal - 2 \tr_{\g} d^{\nabla}\II(N)  - \|\II\|^2_{\g} + (\tr_{\g}
\II)^2. \qquad \qquad \label{scal}
\end{eqnarray*}
for all vectors $X, Y \in T\M$, where $d^{\nabla}\II(X):=
d^{\nabla}\II(X,\cdot,\cdot)$. Hence, we obtain

\begin{Proposition}[Ricci constraint conditions]\label{Prop-Ricci-constriants
spinor}$\,$\\ Let $(\bM,\bg)$ be a Lorentzian spin manifold with a
parallel null spinor field, and let $(\M,\g)$ be a space-like
hypersurface of $(\bM,\bg)$ with normal vector field $T$, second
fundamental form $\II$ and Weingarten operator $\W$. Then for all
vectors $X,Y \in N^{\bot}\subset T\M$,
\begin{eqnarray*}
 \tr_{\g} d^{\nabla}\II(X)  &=& 0,\\
\tr_{\g} d^{\nabla}\II(N) \!\cdot \! N &=& \diver_{\g}\II + \grad_{\g}(\tr_{\g}\II), \\
 2\, \tr_\g d^{\nabla}\II(N)  &=& \scal - \|\II\|_{\g}^2 + (\tr_{\g} \II)^2,
\end{eqnarray*} and   \begin{eqnarray*}
\Ric(X,Y) &=& d^{\nabla}\II(N,X,Y) + \II(\W(X),Y) -
\tr_{\g}\II\cdot\II(X,Y),\\
\Ric(X,N) &=&  \II(\W(X),N) - \tr_{\g}(\II) \cdot\II(X,N), \\
\Ric(N,N) &=& \tr_{\g} d^{\nabla}\II(N) - \tr_{\g}(\II) \cdot \II(N,N) +
\II(\W(N),N).
\end{eqnarray*}
\end{Proposition}
Let $(\overline S,\nabla^{\bS})$ denote the spinor bundle of
$(\bM,\bg)$ with the covariant derivative induced by the Levi-Civita
connection, and let $(S,\nabla^{S})$ be the spinor bundle of the
space-like hypersurface $(\M,\g)$ with its spin derivative. Then
there is a canonical identification of $S$ with $\overline S_{|\M}$
if $n$ is even and of $S$ with the half-spinors $\overline
S^+_{|\M}$ if $n$ is odd. In this identification, the Clifford
product with a vector field $X$ on $\M$ in both bundles is related
via
\be X \cdot \varphi = i \, T \; \bar{\cdot} \; X \; \bar{\cdot} \;
\phi_{|\M}, \label{Id1}\ee where $\varphi \in \Gamma(S)$ is
identified with $\phi_{|\M} \in \Gamma(\overline S^{(+)}_{|\M})$. In
the following we will omit the $\bar{\;\;}$ over the Clifford
multiplication in $\overline S$ in order to keep the notation
simple, it will always be clear in which spinor bundle we are
working. The Dirac current $U_{\psi}$ of a spinor field $\psi$ on a
Riemannian spin manifold $(\M,\g)$ is given by
\belabel{dirac}
 \g(U_{\psi},X) := - i \,(X \cdot \psi,\psi) , \qquad X\in T\M.\eeqn
If $\phi \in \Gamma(\overline S^{(+)})$ is a spinor field on $\bM$
and $\varphi:=\phi_{|\M} \in \Gamma(S)$ its restriction to $\M$, the
Dirac currents satisfies
\[ (V_{\phi})_{|\M} = \|\varphi\|^2 T_{|\M} - U_{\varphi}.\]
Using the above identification of the spinor bundles, the conditions
$\nabla^{\bS} \phi=0$ and $V_{\phi}\cdot \phi = 0$ translate into
the following conditions for the spinor field $\varphi= \phi_{|\M}$:

\begin{Proposition}[Spin constraint conditions]
\label{Prop-Constr-spinor}$\;$\\
If $(\bM,\bg)$ is a spin manifold with a parallel null spinor
field $\phi$. Then the spinor field $\varphi := \phi_{|\M}$ on the
space-like hypersurface $(\M,\g)$ satisfies
 \be \nabla^S_X
\varphi &=& \tfrac{i}{2}\, \W(X)\cdot \varphi \qquad
\forall \; X \in T\M, \label{spinor-1}\\
 U_{\varphi} \cdot \varphi &=& i\,u_{\varphi} \,\varphi, \label{spinor-2}
 \ee
where $\W$ is the Weingarten operator of $(\M,\g)$ and $u_{\varphi}=
\sqrt{\g(U_{\varphi},U_{\varphi})} = \|\varphi\|^2$.
\end{Proposition}

\noindent For a detailed explanation of the identifications used
above and a proof of Proposition \ref{Prop-Constr-spinor} we refer
to \cite{baer-gauduchon-moroianu05} and \cite{baum-mueller08}. For an arbitrary
symmetric $(1,1)$-tensor field $\W$ on a Riemannian spin manifold
$(\M,\g)$ we call a spinor field $\varphi$ on $\M$, satisfying
(\ref{spinor-1}) and (\ref{spinor-2}), an {\em
 imaginary $\W$-Killing spinor}.
 In \cite{baer-gauduchon-moroianu05}  B\"ar, Gauduchon and Moroianu consider the case of a  {\em real $W$-Killing spinor} $\vf$ on a semi-Riemannian
manifold $(\M,\g)$. They show that, if $W$ is a
Codazzi tensor,  the manifold $(\M,\g)$ can be embedded as a
hypersurface into a Ricci flat manifold $(\bM,\bg=dr^2+\g_r)$ equipped with a parallel spinor
which restricts to $\varphi$.

We will now show, that any imaginary $\W$-Killing spinor $\varphi$
on a Riemannian spin manifold $(\M,\g)$ can be extended to a
parallel light-like spinor field $\phi$ on an open neighbourhood
$\overline{\cU}(\{0\} \times \M) \subset \rr \times \M$ with a
Lorentzian metric $\bg:= - \lambda^2 dt^2 + \g_t$ such that
$\phi_{|\M} = \varphi$, at least if all given data are real
analytic. First we show, that the Dirac current
$U_{\varphi}$ of $\varphi$ satisfies the constraint conditions
(\ref{constraintU2})  of Theorem
\ref{Theorem-parallel-vf2}.

\begin{Lemma}\label{Lemma-Killingspinor-1} Let $\varphi$ be an
imaginary $\W$-Killing spinor on $(\M,\g)$. Then the Dirac current
$U_{\varphi}$ of $\varphi$ satisfies
\be X(u_{\varphi})&=& -\g(\W(X),U_{\varphi}), \label{spinor-3}\\
   \nabla_X U_{\varphi} &=& - u_{\varphi} \W(X) \label{spinor-4}
   \ee
for all vector fields $X$ on $\M$. \end{Lemma}
\begin{proof} We write  $u:= u_{\varphi}$ and $U:=U_{\varphi}$.
Since $u = (\varphi,\varphi)$, we obtain
%
\beq X(u) &=& (\nabla^S_X \varphi, \varphi ) + 
\overline{ (\nabla^S_X \varphi, \varphi ) }
\\
& =& \tfrac{i}{2} (\W(X)\cdot \varphi, \varphi)
+\overline{ \tfrac{i}{2} (\W(X)\cdot \varphi, \varphi)}
\\
&=& - \g(\W(X),U), \eeq
by equation \eqref{dirac}. Differentiating (\ref{spinor-2}) and
inserting (\ref{spinor-1}) and (\ref{spinor-3}) gives
 \beq \nabla^S_X (U\cdot \varphi) &=& \nabla_XU \cdot \varphi + U \cdot
 \nabla^S_X\varphi \\
 &=& \nabla_XU \cdot \varphi + \tfrac{i}{2}U\cdot  \W(X)\cdot
 \varphi \\
 &=& \nabla_XU \cdot \varphi - \tfrac{i}{2} \W(X)\cdot U \cdot
 \varphi - i \g(U,\W(X)) \varphi \\
 &=& \nabla_XU \cdot \varphi + \tfrac{1}{2}u \W(X)\cdot
 \varphi - i \g(U,\W(X)) \varphi\; , \\
 \nabla^S_X (iu\varphi) &=& i X(u)\varphi + iu \nabla^S_X\varphi \\
 &=& -i
 \g(\W(X),U)\varphi
- \tfrac{1}{2} u \W(X)\cdot \varphi\;.
 \eeq
 Hence, $\; (\nabla_XU + uW(X)) \cdot \varphi = 0\,$, which shows
 (\ref{spinor-4}).
\end{proof}

Now, let us suppose, that the Riemannian spin manifold $(\M, \g)$
and the field of $\g$-symmetric endomorphisms $\W$ are real
analytic. Then the Dirac current $U_{\varphi}$ of an imaginary
$\W$-Killing spinor $\varphi$ and its length $u_{\varphi}$ are real
analytic as well and satisfy the constraint equations of Theorem
\ref{Theorem-parallel-vf2} and Corollary \ref{Cor-KWparallel}.
Hence, for any positive analytic function $\lambda$ on $\rr \times
\M$ there exists an open neighbourhood $\overline{{\cal U}}$ of $\M
\simeq \{0\} \times \M \subset \rr \times \M$ and an unique analytic
Lorentzian metric $\bg= - \lambda^2 \,dt^2 + \g_t$ on
$\overline{{\cal U}}$, which admits an analytic parallel null vector
field $V = \frac{u_t}{\lambda} \del_t - U_t$, where $(\g_t,U_t,u_t)$
are solutions of the evolution equations (\ref{Vevolvg1sys2},
\ref{dtdtu1sys2}, \ref{dtdtu2sys2}) with initial conditions
\eqref{icg02}, where $U:=U_{\varphi}$ and $u:=u_{\varphi}$. Now we
observe the following property of the parallel null vector field $V=
\frac{u_t}{\lambda} \del_t - U_t$.

\begin{Lemma}
Let $\phi \in \Gamma(\overline S^{(+)})$ be defined by parallel
transport of $\varphi \in \Gamma(S \simeq \overline
S^{(+)}_{|\M})$ along the $t$-lines $\gamma_x(t)=(t,x)$ of
$\overline{\cU}$. Then $V$ is the Dirac current of $\phi$.
\end{Lemma}
\begin{proof}
We use the global vector fields $T$ and $N:=u^{-1}U$ to reduce the
frame bundle of $(\overline{\cU},\bg)$ to the subgroup $\mathbf{SO}(n-1)
\subset \mathbf{SO}(1,n)$. Then, the spin structure of $(\overline{\cU},\bg)$
is given by a spin structure $\widehat Q$ of the reduced frame
bundle $\widehat P$. Let $\widehat S := \widehat Q
\times_{Spin(n-1)} \Delta_{n-1}$. Since the spinor modul
$\Delta_{1,n}$ is isomorphic to $\Delta_{n-1}\otimes \Delta_{1,1}$,
we can identify the spinor bundle $\overline S$ of
$(\overline{\cU},\bg)$ with $\widehat S  \otimes \Delta_{1,1}$,
where $T$, $N$ and $X \in \mathrm{span}(T,N)^{\bot}$ act on
$\widehat \psi \otimes u(\varepsilon) \in \widehat{S} \otimes
\Delta_{1,1}$ by
\beq T \cdot (\widehat{\psi} \otimes u(\varepsilon))
&=& - \widehat{\psi} \otimes u(-\varepsilon),\\
N \cdot (\widehat{\psi} \otimes u(\varepsilon)) &=& \varepsilon \,
\widehat{\psi} \otimes u(-\varepsilon),\\
X  \cdot (\widehat{\psi} \otimes u(\varepsilon)) &=& - \varepsilon
\,(X \cdot \widehat{\psi}) \otimes u(\varepsilon) .\eeq
Here,  $\big\{ u(\varepsilon) := {1 \choose -\varepsilon i} \mid
\varepsilon = \pm 1 \big\}$ denotes an unitary basis in the complex
vector space $\Delta_{1,1} \simeq \CC^2$. Then for $\psi=
 \widehat \psi_{1} \otimes
u(1) + \widehat \psi_{-1} \otimes u(-1) \in \Gamma(\overline{S})$ it holds:
\be V \cdot \psi = 0 & \;  \Longleftrightarrow \; & T \cdot \psi = N
\cdot \psi \;\; \Longleftrightarrow \;\; \psi = T \cdot N \cdot \psi
\; \; \Longleftrightarrow \; \; \psi = \widehat{\psi}_{-1} \otimes
u(-1). \label{Id2}\ee
Now, since $U\cdot \varphi = iu\varphi$ in $\Gamma(S)$, the spinor
field $\,\phi_{|\M} \in \Gamma(\overline S^{(+)}_{|\M})\,$ satisfies
$i T \cdot U \cdot \phi_{|\M} = i u \phi_{|\M}$ (see (\ref{Id1}))
and therefore, $T\cdot N \cdot \phi_{|\M}= \phi_{|\M}$ . This shows,
that $(V\cdot \phi)_{|\M} = 0$. But $V$ as well as $\phi$ are
parallel along the $t$-lines and we obtain
\[ \nabla^{\bS}_{\partial_t}(V\cdot \phi) = \bnab_{\partial_t}V \cdot
\phi + V \cdot \nabla^{\bS}_{\partial_t}\phi = 0. \] Thus the spinor
field $V \cdot \phi$ is parallel along the $t$-lines as well. Since
it vanishes on $\{0\}\times \M$, it vanishes on $\overline{\cU}$.
Then for the Dirac current of $\phi$ hold
\beq \bg(V_{\phi}, T) &=& - \langle T \cdot \phi,\phi \rangle = - (T
\cdot T \cdot \phi , \phi ) = - (\phi, \phi), \\
\bg(V_{\phi}, N) &=& - \langle N \cdot \phi,\phi \rangle = - (T
\cdot N \cdot \phi , \phi ) = - (\phi, \phi),\\
\bg(V_{\phi}, X) &=& - \langle X \cdot \phi,\phi \rangle = - (T
\cdot X \cdot \phi , \phi ) =  (\,(X \cdot \widehat{\phi}_{-1}
\otimes u(1)\,,\, \widehat{\phi}_{-1} \otimes u(-1)\, ) = 0 \eeq for
$X \in \mathrm{span}(T,N)^{\bot}$. This shows, that $V =
\tfrac{u}{\|\phi\|^2} V_{\phi}$. Since $V$ and $V_{\phi}$ are
parallel along the $t$-lines, $u(t,x)= c(x)\|\phi(t,x)\|^2$. Because
of $\,u(0,x)=\|\varphi(x)\|^2 = \|\phi(0,x)\|^2$ we have $c(x)=1$.
This shows that $V$ is the Dirac current of $\phi$.
\end{proof}

Using a similar method as the authors of \cite{AmmannMoroianuMoroianu13}, 
we will now show that $\phi$ is parallel on $(\overline{\cU},\bg)$.

\begin{Theorem}\label{Theorem-Extension spinor}
Let $(\M,\g)$ be an analytic Riemannian spin manifold with an
analytic $\g$-symmetric endomorphism field $\W$ and $\varphi$ an
imaginary $\W$-Killing spinor on $(\M,\g)$, and let
\[\big(\,\overline{\cU}(\{0\} \times \M)\,,\, \bg= -\lambda^2 dt^2 +
\g_t\big)\] be the Lorentzian manifold with  parallel null vector field $V$,
arising as the solutions of the evolution equations
(\ref{Vevolvg1sys2})-(\ref{dtdtu2sys2}) in Corollary \ref{Cor-KWparallel} with the initial conditions
(\ref{icg02}) given by
$(\M,\g,\W,U_{\varphi})$. Let $\phi$ be the spinor field on
$(\overline{\cU},\bg)$ obtained by parallel transport of
$\varphi$ along the $t$-lines $t\mapsto (t,x)$. Then $\phi$ is a parallel spinor
field on $(\overline{\cU},\bg)$ with the Dirac current
$V_{\phi}=V$.
\end{Theorem}
\begin{proof} Since $\nabla^{\bS}_{\partial_t} \phi = 0$ by definition,  it remains to
show, that $\nabla^{\bS}_X\phi=0$ for all vector fields $X$ on
$\overline{\cU}$ tangent to $\M$. In the following we will consider
the bundle $\Lambda^kT^*\M\! \otimes \bS$ of $k$-forms on
$T\M$ with values in $\bS$ with the covariant derivative
$\bnab$ induced by the Levi-Civita connection of $\bg$ and the spin
connection $\nabla^{\bS}$:
\beq \bnab_X(\omega)(Y_1,\ldots,Y_k)&:=&
\nabla_X^{\bS}\big(\omega(Y_1,\ldots,Y_k)\big) -
\sum\limits_{i=1}^{k} \omega(Y_1,\ldots,Y_{i-1},\prT \bnab_X Y_i,
Y_{i+1},\ldots,Y_k) \eeq
 for $\omega \in
\Gamma(\Lambda^kT^*\M\! \otimes \bS)$.\\
We consider now the section ${A \choose B}$ of the bundle $ E:=
(T^*\M \!\otimes \overline S) \oplus (\Lambda^2T^*\M\!
\otimes \overline S)$  over $\overline{\cU}$, defined by
\beq A(X) &:=& \nabla^{\bS}_X \phi, \\
     B(X,Y) &:=& \R^{\overline S}(X,Y)\phi
\eeq for $X,Y \in T\M$. In order to show that $A=0$, we show
that ${A \choose B}$ satisfies the following linear PDE on $E$
\begin{eqnarray} \bnab_{\partial_t} {A \choose B} = Q {A \choose B},
\label{linPDE-spinor} \end{eqnarray} where $Q$ is the linear
operator on $E$, given by
\[ Q {\omega \choose \mu} = \lambda \left(%
\begin{array}{c}
 + \mu(N,\cdot) + \omega \circ \W \\[0.1cm]
  - d^{\bnab}\mu(\cdot ,N) + d(\ln \tfrac{\lambda}{u})\wedge \mu(N, \cdot) + \tfrac{1}{2}\bR(N, \cdot)\wedge \,\omega
\end{array}%
\right).
\]

\noindent To this aim, we use the following formula for the
curvature of the spin connection in $\overline S$ (see
\cite{BFGK91}):
\[ \R^{\overline S}(X,Y)\psi = \tfrac{1}{2} \bR(X,Y) \cdot \psi, \]
where on the right hand side stands the Clifford multiplication of
$\psi$ with the 2-form 
\[ \bR(X,Y) = \sum\limits_{i< j}  \bR(X,Y,\cdot,\cdot) .\] 
Since $\bnab V=0$ and $V= u(T - N)$, the curvature of $\bg$
satisfies $\bR(T,X,Y,Z) = \bR(N,X,Y,Z)$ for all vectors $X, Y,Z \in
T\overline{\cU}$. Hence,
\be \R^{\overline S}(T,X)\psi &=& \R^{\overline S}(N,X)\psi \qquad
\forall \; X \in T\M. \label{spinor-prolong1} \ee Furthermore we
obtain
\be \bnab_X N = X(\ln u) T - X(\ln u) N - \W(X) \qquad \forall \; X
\in T\M.\label{spinor-prolong2} \ee Moreover, \be
B(\prT\!\bnab_TX,Y) &=& \R^{\bS}(\bnab_TX, Y)\phi - X(\ln
\lambda)\R^{\bS}(T,Y)\phi \nonumber \\
 &=& \R^{\bS}(\bnab_TX, Y)\phi - X(\ln \lambda)B(N,Y) \label{spinor-prolong3}   \ee
for all vector fields $X,Y \in \Gamma(T^\perp)$, since $\, \bnab_T
X = X(\ln \lambda))\, T + \prT \bnab_T X$.\\
Then, using (\ref{spinor-prolong1}), (\ref{spinor-prolong2}),
(\ref{spinor-prolong3}) and the second Bianchi identity for $\bR$,
we obtain for all vector fields $X,Y \in \Gamma(T\M)$
\beq (\bnab_{T}A)(X) &=& \nabla^{\bS}_T(A(X))  - A(\prT \bnab_TX) \\
&=& \nabla^{\bS}_T\nabla^{\bS}_X\phi - \nabla^{\bS}_{\bnab_TX}\phi
\\
&=& R^{\bS}(T,X)\phi + \nabla_X^{\bS}\nabla_T^{\bS}\phi +
\nabla_{[T,X]}^{\bS}\phi - \nabla_{\bnab_T X}^{\bS}\phi \\
&=& R^{\bS}(N,X)\phi - \nabla^{\bS}_{\bnab_XT}\phi \\
&=& B(N,X) + A(\W(X)) \eeq
and
 \beq (\bnab_T B) (X,Y) &=& \nabla^{\bS}_T(R^{\bS}(X,Y)\phi) -
 B(\prT\bnab_TX,Y) - B(X,\prT \bnab_T Y) \\
 &=& \tfrac{1}{2}\bnab_T(\bR(X,Y))\cdot \phi  -
 \tfrac{1}{2}\bR(\bnab_TX,Y) \cdot\phi -
 \tfrac{1}{2}\bR(X,\bnab_TY)\cdot \phi \\
 && + X(\ln \lambda) B(N,Y) - Y(\ln \lambda) B(N,X) \\
 &=& \tfrac{1}{2}(\bnab_T \bR)(X,Y)\cdot \phi + X(\ln \lambda) B(N,Y) - Y(\ln \lambda) B(N,X) \\
 &=& +\tfrac{1}{2} (\bnab_X \bR)(T,Y) \cdot \phi -
\tfrac{1}{2}(\bnab_Y \bR)(T,X)\cdot\phi + X(\ln \lambda) B(N,Y) - Y(\ln \lambda) B(N,X) \\
&=& \nabla^{\bS}_X(\R^{\bS}(T,Y)\phi) - \tfrac{1}{2}\bR(T,Y) \cdot
\nabla^{\bS}_X\phi - \R^{\bS}(\bnab_XT,Y)\phi -
\R^{\bS}(N,\prT\!\bnab_XY)\phi \\
&& - \nabla^{\bS}_Y(\R^{\bS}(T,X)\phi) + \tfrac{1}{2}\bR(T,X) \cdot
\nabla^{\bS}_Y\phi + \R^{\bS}(\bnab_YT,X)\phi +
\R^{\bS}(N,\prT\!\bnab_Y
X)\phi\\
&& + X(\ln \lambda) B(N,Y) - Y(\ln \lambda) B(N,X)\\
&=& (\bnab_X B)(N,Y) - \tfrac{1}{2}\bR(N,Y)A(X) + B(\W(X),Y) + B(\prT\!\bnab_XN,Y) \\
&&  - (\bnab_Y B)(N,X) + \tfrac{1}{2}\bR(N,X)A(Y) - B(\W(Y),X) + B(\prT\!\bnab_YN,X) \\
&& + X(\ln \lambda) B(N,Y) - Y(\ln \lambda) B(N,X)\\
&=& (\bnab_X B)(N,Y) - \tfrac{1}{2}\bR(N,Y)A(X) - X(\ln u) B(N,Y)  \\
&&  - (\bnab_Y B)(N,X) + \tfrac{1}{2}\bR(N,X)A(Y) + Y(\ln u) B(N,X) \\
&& + X(\ln \lambda) B(N,Y) - Y(\ln \lambda) B(N,X)\\
 &=& (\bnab_X B)(N,Y) - (\bnab_Y B)(N,X) + X(\ln \tfrac{\lambda}{u}) B(N,Y) - Y(\ln
\tfrac{\lambda}{u}) B(N,Y) \\
&&  - \tfrac{1}{2}\bR(N,Y)\cdot A(X) + \tfrac{1}{2}\bR(N,X)\cdot
A(Y).  \eeq
This shows, that the section ${A \choose B}$ in $E$ solves the
linear PDE (\ref{linPDE-spinor}). Now, let us consider the section
${A \choose B}$ on the initial hypersurface $ \{0\}\times \M \simeq
\M$. First, we have
\be A(X)_{|t=0} = \bnab_X \phi\,_{|t=0} = 0,
\label{spinor-prolong4}\ee
for all $X \in T\M$, since $\phi$ is defined by parallel
transport of the imaginary Killing spinor $\varphi$ and
(\ref{spinor-prolong4}) is the $\overline{S}_{|\M}$ correspondence
for the Killing condition (\ref{spinor-1}) of $\varphi \in
\Gamma(S)$. Then of course, $B(X,Y)_{|t=0}= \R^{\overline
S}(X,Y)\phi_{|t=0}= 0$ for all $X,Y \in T\M$. \\
To summerize, the section  ${A \choose B} \in \Gamma(E)$ solves the
linear PDE (\ref{linPDE-spinor}), where the time derivative
$\bnab_{\partial_t}$ separates on the left hand side and on the
right hand side are only derivatives in the space-direction of
maximally first order, with initial condition ${A \choose B}|_{t=0}
= 0$ on the initial space-like hypersurface $\M \simeq \{0\}\times
\M$. Since all data are real analytic, the Cauchy-Kowalevski Theorem
guaranties
 an unique analytic solution. Hence, by the initial
conditions, this solution is identically zero. This shows
that the spinor field $\phi$ on $(\overline{\cU},\bg)$ is parallel.
\end{proof}

\bbem
Our method needs analyticity of all data. In the Riemannian analogue, smooth
real $\W$-Killing spinors in general cannot be extended to parallel spinors
(see \cite{bryant10} and \cite{AmmannMoroianuMoroianu13}). 
Having the analogous situation for the Einstein equation in mind, where the work of Choquet-Bruhat \cite{Foures-Bruhat52} dealt with the smooth case,   the question remains whether in the Lorentzian setting smoothness is sufficient for extending $\W$-Killing spinors to parallel ones, or if there are smooth examples that cannot be extended.
\ebem



\section{Riemannian manifolds satisfying the constraint equations}
\label{examples}

In this section we describe some examples of {\em complete}
Riemannian manifolds $(\M,\g,\W,U)$ satisfying the constraint
conditions (\ref{constraintU2})  of Theorem
\ref{Theorem-parallel-vf2} and admitting imaginary $\W$-Killing
spinors.
First recall that for a
vector field $U$ the endomorphism field  $\nabla U$ is  symmetric with respect to $\g$  
 if and only if the metric dual  $U^\flat=\g(U,.) $ is a  closed one form.
This implies that a Riemannian manifold $(\M,\g,\W,U)$ satisfying the
contraint conditions (\ref{constraintU2}) 
of Theorem~\ref{Theorem-parallel-vf2} is foliated in integral
manifolds $\F$ of $U^{\bot}=\mathrm{Ker}(U^\flat)$ with the second fundamental form
$\II^{\F}(X,Y) = \II(X,Y)$ for all $X,Y \in U^{\bot}$ and the
Weingarten operator $\W^{\F}:= \pr_{U^\bot} \circ \W$.

Now we collect some integrability conditions for an
imaginary $\W$-Killing spinor.

\begin{Lemma}\label{Lemma-Constr-spinor} Let $(\M,\g)$ be a
Riemannian spin manifold with a symmetric endomorphism field $W$ and
suppose that there is a (non-trivial) imaginary $\W$-Killing spinor
$\varphi$ as in \eqref{spinor-1} and \eqref{spinor-2}
with Dirac current $U_{\varphi}$ of $\varphi$ and $u_{\varphi}= \|\varphi\|^2 = \sqrt{\g\big(U_{\varphi},U_{\varphi} \big)} >0$. 
Then $U_{\varphi}^\flat $ is a closed $1$-form and the integral
manifolds of the distribution $U_{\varphi}^{\bot}$ are Riemannian
spin manifolds with a parallel spinor field. Moreover, $U_{\varphi}$
and $u_{\varphi}$ satisfy
\begin{eqnarray*}
\nabla_X U_{\varphi}
&=& - u_{\varphi}\W(X), \\
 \W(U_{\varphi})&=& \grad \,u_{\varphi}, \\
 \g(d^{\nabla}\!\W(X,Y), U_{\varphi})&=& 0, 
\\
 X(u_{\varphi})&=& - \g(\W(X),U_{\varphi}), \\
YX(u_{\varphi}) &=& -\g\big((\nabla_Y\W)(X)+
\W(\nabla_YX),U_{\varphi}\big) + u_{\varphi}
\g\big(\W(X),\W(Y)\big), \\
\Hess(u_{\varphi})(X,Y) &=& -\g\big(
(\nabla_X\W)(Y),U_{\varphi}\big) + u_{\varphi}
\g\big(\W(X),\W(Y)\big).  
\end{eqnarray*}
for  $X, Y\in T\M$. The Dirac operator $D$ and the
Bochner-Laplace operator $\nabla^* \nabla$ on the spinor bundle
applied to $\varphi$ yield
 \begin{eqnarray*} D\varphi &=& - \tfrac{i}{2} \tr_{\g}(\W) \varphi, \label{spinor-6}\\
     D^2 \varphi &=& - \tfrac{1}{4}(\tr_{\g}(\W))^2 \varphi - \tfrac{i}{2}
     \grad(\tr_{\g}(\W))\cdot \varphi, \label{spinor-7}\\
     \nabla^*  \nabla \varphi &=& \tfrac{i}{2} \diver_{\g}(\W) \cdot
     \varphi - \tfrac{1}{4} \tr_{\g}(\W^2) \varphi. \label{spinor-8}
\end{eqnarray*}
Furthermore, the curvature of the spin connection satisfies
\be \R^S(X,Y)\varphi &=& \tfrac{i}{2} d^{\nabla}\!\W(X,Y) \cdot
\varphi +
\tfrac{1}{4} (\W(X)\cdot \W(Y) - \W(Y)\cdot \W(X)) \cdot \varphi, \label{spinor-9} \nonumber\\
\Ric(X) \cdot \varphi &=& -i \sum_{j=1}^n s_j \wedge
d^{\nabla}\!\W(X,s_j) \cdot \varphi - i\, \tr_{\g}  d^{\nabla}\II(X)(X) \varphi -
\tr_{\g}(\W) \W(X)\cdot \varphi + \W^2(X)\cdot \varphi.
\label{spinor-10} \nonumber \ee
\end{Lemma}
\begin{proof}
The proof is a straightforward calculation in spin geometry, completely analogous to the one  carried out for imaginary Killing spinors in \cite{baum89-3} and
for real $\W$-Killing spinors in \cite{MoroianuSemmelmann14},  for example.
\end{proof}
Finally we describe three classes of examples for complete
Riemannian spin manifolds $(\M,\g)$ with imaginary $\W$-Killing
spinors.

\begin{bsp}
 Let $(\M,\g)$ be compact and 2-dimensional.
Since we are looking for compact 2-dimensional manifolds with a
nowhere vanishing vector field (the Dirac current of the
$\W$-Killing spinor), it is enough to restrict ourself to the
2-torus $T^2$ equipped with a metric $\g$ conformally equivalent to
the flat metric $\g_0$,
\[ \g := e^{2 \sigma} \g_0 = e^{2 \sigma}\big(a \,d\theta^2 + 2b \,d\theta
d\rho + c \,d\rho^2 \big), \qquad a,c \in \rr^+, b\in \rr, ac > b^2.
\]
Let $U:= f \del_{\theta} + h \del_{\rho}$ be a vector field on $T^2$
and suppose that $U$ has no zeros, $a f^2 + 2b fh + c\, h^2
>0$. Then $U$ is closed, respectively the endomorphism field
\[ \W:= -
\frac{1}{\|U\|_{\g}}\nabla^{\g} U \] is $\g$-symmetric, if and only
the functions $\,\sigma, f, h \in C^{\infty}(T^2,\rr)$ satisfy the
PDE
\begin{eqnarray}
( b \del_{\theta} - a \del_{\rho})\big(e^{2\sigma}f\big)  & = &
 (-c \del_{\theta} + b \del_{\rho})\big(e^{2\sigma} h
 \big).
 \label{Beispiel-1.1} \nonumber
\end{eqnarray}
We choose on $(T^2,g)$ the trivial spin structure and consider the
spinor field $\varphi:=\gamma \cdot v \in \Gamma(S) \simeq
C^{\infty}(T^2, \Delta_2)$, where $\,\gamma \in
C^{\infty}(T^2,\C)\,$ is a function with $\,|\gamma|^2 =
\|U\|_{\g}\,$ and $\,v \in \Delta_2\,$ is a fixed spinor with $e_1
\cdot v = iv\,$ and $\,\|v\|=1$. A direct calculation shows, that
$\varphi$ is an imaginary $\W$-Killing spinor on $(T^2,
e^{2\sigma}g_0)$ with Dirac current $U$. Morover, {\em all}
$\,\g$-symmetric endomorphisms $\W$ and imaginary $\W$-Killing
spinors on the torus $(T^2,e^{2\sigma}g_0)$ with trivial spin
structure are of this form.
\end{bsp}

\begin{bsp}
Let $(\cF,\g_{\cF})$ be a complete Riemannian spin manifold with a
parallel spinor field,  a Codazzi tensor  $\T$ on $(\cF,\g_{\cF})$ and
$b\in C^{\infty}(\rr,\rr)$ a smooth function. Let $\A$ be the
$(1,1)$-tensor field on $\M:= \rr \times \cF$ given by
\[ \A = \left(%
\begin{array}{cc}
  b(s)\cdot \mathrm{Id}_{T\rr}  & 0 \\
  0 &    e^{s} \Big(\T - \int\limits_0^s b(r)e^{-r}dr \cdot
  \mathrm{Id}_{\T\cF}\Big)
\end{array}%
\right).\] Then $\g:= \A^*(ds^2 + e^{-2s}\g_{\cF})$ is a complete
Riemannian metric on $\M:= \rr \times \cF$, $\W:= \A^{-1}$ is an
invertible Codazzi tensor on $(\M,\g)$ and $(\M,\g)$ equipped with
the spin structure induced by that of $(\cF,\g_{\cF})$ admits an
imaginary $\W$-Killing spinor. On the other hand, all complete
Riemannian spin manifolds $(\M,\g)$ with imaginary $\W$-Killing
spinor for an invertible Codazzi tensor $\W$ arise in this way. (For
a proof see \cite{baum-mueller08}).
\end{bsp}

\begin{bsp}
Let $(\M,\g)$ be a Riemannian manifold and $\W:= b \, \Id_{T\M}$,
where $b$ is a smooth function on $\M$ which is not identically
zero. Then $\W$ is in general neither a Codazzi tensor nor
invertible. Suppose, that $U$ is a vector field and $u$ a positive
smooth function on $\M$ such that $\,\nabla U = - u \W\,$ and
$\,\g(U,U) = u^2
>0$. Then,
\[ \cL_U\g(X,Y) = \g(\nabla_XU,Y) + \g(X,\nabla_Y U) = -2b \g(X,Y), \]
hence $U$ is a conformal vector field on $(\M,\g)$ without zeros,
which is closed as $\nabla U$ is symmetric. Therefore $(\M,\g)$
is isometrically covered by the warped product $(\rr \times \cF,
ds^2 + h(s)^2\g_{\cF})$, where $(\cF,\g_{\cF})$ is a complete
Riemannian manifold, $h\in C^{\infty}(\rr,\rr^+)$ and $b$, $U$ and
$u$ satisfy \beq b(\pi(s,x)) &=& -
(\ln{h})'(s),\\
U(\pi(s,x)) &=& d \pi\big( h(s)\partial_s(s,x)\big),\\
u(\pi(s,x)\big) &=& h(s), \eeq where $\pi$ denotes the covering map.
 If $(\M,\g)$ is in addition spin and $U$ the Dirac
current of an imaginary $\W=b\,\Id_{T\M}$--Killing spinor, then
$(\cF,\g_{\cF})$ is spin as well with a parallel spinor field.
Conversely, any warped product $\M:= L\times_h \cF$, where
$(\F,\g_{\F})$ is a complete Riemannian manifold, $L\in \{S^1,\rr
\}$ and $h$ is a smooth positive function on $L$ admits the closed
conformal vector field $U(s,x) = h(s) \partial_s(s,x)$ of length
$u=h$, and the endomorphism $\W:= b \,\Id_{T\M}$ with $b:= -
\ln(h)'$ satisfies $\nabla U = -u\W$. If $(\cF,\g_{\cF})$ is spin
and has a parallel spinor field, then $\M= L\times_h \cF$ is spin as
well with an imaginary $\W= b\,\Id_{T\M}$--Killing spinor. For a
proof of all these statements see \cite{rademacher91}.
\end{bsp}


\begin{thebibliography}{10}

\bibitem{AmmannMoroianuMoroianu13}
B.~Ammann, A.~Moroianu, and S.~Moroianu.
\newblock The {C}auchy problems for {E}instein metrics and parallel spinors.
\newblock {\em Comm. Math. Phys.}, 320(1):173--198, 2013.

\bibitem{baer-gauduchon-moroianu05}
C.~B{{\"a}}r, P.~Gauduchon, and A.~Moroianu.
\newblock Generalized cylinders in semi-{R}iemannian and {S}pin geometry.
\newblock {\em Math. Z.}, 249(3):545--580, 2005.

\bibitem{BartnikIsenberg04}
R.~Bartnik and J.~Isenberg.
\newblock The constraint equations.
\newblock In {\em The {E}instein equations and the large scale behavior of
  gravitational fields}, pages 1--38. Birkh\"auser, Basel, 2004.

\bibitem{baum89-3}
H.~Baum.
\newblock Complete {R}iemannian manifolds with imaginary {K}illing spinors.
\newblock {\em Ann. Global Anal. Geom.}, 7(3):205--226, 1989.

\bibitem{BFGK91}
H.~Baum, T.~Friedrich, R.~Grunewald, and I.~Kath.
\newblock {\em Twistors and {K}illing spinors on {R}iemannian manifolds},
  volume 124 of {\em Teubner-Texte zur Mathematik [Teubner Texts in
  Mathematics]}.
\newblock B. G. Teubner Verlagsgesellschaft mbH, Stuttgart, 1991.
\newblock With German, French and Russian summaries.

\bibitem{baum-mueller08}
H.~Baum and O.~M{{\"u}}ller.
\newblock Codazzi spinors and globally hyperbolic manifolds with special
  holonomy.
\newblock {\em Math. Z.}, 258(1):185--211, 2008.

\bibitem{bb-ike93}
L.~B{\'e}rard-Bergery and A.~Ikemakhen.
\newblock On the holonomy of {L}orentzian manifolds.
\newblock In {\em Differential Geometry: Geometry in Mathematical Physics and
  Related Topics (Los Angeles, CA, 1990)}, volume~54 of {\em Proc. Sympos. Pure
  Math.}, pages 27--40. Amer. Math. Soc., Providence, RI, 1993.

\bibitem{bernal-sanchez03}
A.~N. Bernal and M.~S{\'a}nchez.
\newblock On smooth {C}auchy hypersurfaces and {G}eroch's splitting theorem.
\newblock {\em Comm. Math. Phys.}, 243(3):461--470, 2003.

\bibitem{bryant10}
R.~L. Bryant.
\newblock Non-embedding and non-extension results in special holonomy.
\newblock In {\em The many facets of geometry}, pages 346--367. Oxford Univ.
  Press, Oxford, 2010.

\bibitem{Folland95}
G.~B. Folland.
\newblock {\em Introduction to partial differential equations}.
\newblock Princeton University Press, Princeton, NJ, second edition, 1995.

\bibitem{Foures-Bruhat52}
Y.~Four{\`e}s-Bruhat.
\newblock Th\'eor\`eme d'existence pour certains syst\`emes d'\'equations aux
  d\'eriv\'ees partielles non lin\'eaires.
\newblock {\em Acta Math.}, 88:141--225, 1952.

\bibitem{galaev05}
A.~S. Galaev.
\newblock Metrics that realize all {L}orentzian holonomy algebras.
\newblock {\em Int. J. Geom. Methods Mod. Phys.}, 3(5-6):1025--1045, 2006.

\bibitem{Koiso81}
N.~Koiso.
\newblock Hypersurfaces of {E}instein manifolds.
\newblock {\em Ann. Sci. \'Ecole Norm. Sup. (4)}, 14(4):433--443 (1982), 1981.

\bibitem{leistnerjdg}
T.~Leistner.
\newblock On the classification of {L}orentzian holonomy groups.
\newblock {\em J. Differential Geom.}, 76(3):423--484, 2007.

\bibitem{MoroianuSemmelmann14}
A.~Moroianu and U.~Semmelmann.
\newblock Generalized {K}illing spinors on spheres.
\newblock {\em Ann. Global Anal. Geom.}, 46(2):129--143, 2014.

\bibitem{rademacher91}
H.-B. Rademacher.
\newblock Generalized {K}illing spinors with imaginary {K}illing function and
  conformal {K}illing fields.
\newblock In {\em Global differential geometry and global analysis ({B}erlin,
  1990)}, volume 1481 of {\em Lecture Notes in Math.}, pages 192--198.
  Springer, Berlin, 1991.

\end{thebibliography}
\def\cprime{$'$} \def\cprime{$'$} \def\cprime{$'$}

\end{document}